\documentclass{article}

\usepackage[letterpaper]{geometry}
\usepackage{setspace}
\usepackage{lmodern} % Use the lmodern font

\usepackage{amsmath, amsfonts, amsthm, amsbsy, mathtools, mathrsfs}
 \usepackage{amssymb}

\usepackage{tikz}
\usetikzlibrary{positioning, arrows.meta, shapes.geometric, fit, calc, backgrounds}

 \usepackage{float}
\newtheorem{lemma}{Lemma}[section]
\newtheorem{theorem}{Theorem}[section]

\newtheorem{assumption}{Assumption}[section]

\usepackage{graphicx}
\usepackage{xcolor}
\usepackage{tikz}
\usetikzlibrary{shapes,arrows.meta,positioning}
\usetikzlibrary{arrows.meta}
\usepackage{listings}
\usepackage{algorithm}
\usepackage{algpseudocode}
\usepackage{caption}
\usepackage{subcaption}
\usepackage{cprotect}
\usepackage{dsfont}
\usepackage{tabularx}
\usepackage{array}
\usepackage{multirow}
\usepackage{booktabs}
\usepackage{siunitx}

\usepackage[numbers]{natbib}

\usepackage{authblk}

\usepackage{csvsimple}
\usepackage[normalem]{ulem}
\usepackage{lipsum} % For dummy text
\usepackage{hyperref}

\title{A Multi-Level Machine Learning Framework  for Inverse Scattering Problems with Multi-Frequency Data}
\author{Yi Liu, Yanzhao Cao, Junshan Lin, Yimin Zhong\\
\smallskip
Department of Mathematics and Statistics, Auburn University}
\date{}

\begin{document}
\maketitle

\begin{abstract}
In this work, we propose a multi-level machine learning framework for solving inverse scattering problems with multi-frequency data. The multi-level neural network is built along the frequency axis of the scattering problem, wherein at each fixed frequency, a new level of network is added to the existing architecture to update the reconstruction. By marching through the frequency levels, the proposed multi-level computational framework is able to obtain higher-order Fourier modes of the imaging target as the depth of the neural network grows and higher-frequency data are used. Furthermore, the overall learning problem is decomposed into a sequence of simpler local tasks, each associated with a single frequency. This decomposition significantly reduces the complexity of the optimization problem and mitigates the risk of convergence to undesirable local minima, resulting in a robust and reliable training procedure for solving inverse scattering problems. We conduct various numerical experiments for the inverse source scattering problem and the inverse medium scattering problem to illustrate the effectiveness and robustness of the proposed machine learning framework. In addition, theoretical analysis in the neural tangent kernel regime shows that the proposed multi-level architecture progressively recovers the higher-order Fourier components of the imaging target.
\end{abstract}


\section{Introduction}
Inverse scattering problems seek to recover unknown scattering objects or radiating sources from measurements of acoustic or electromagnetic waves in the near or far field. They arise in a wide range of applications, such as medical imaging, remote sensing, nondestructive testing, and geophysical exploration \cite{colton_2019, Isakov2017, Uhlmann2014}.
These problems are challenging in realistic applications because they are generally ill-posed. Furthermore, in many cases, the forward map is nonlinear, which further complicates the reconstruction process.

Over the years, a variety of computational methods have been developed for solving inverse scattering problems. These methods can be grouped into two categories: optimization-based iterative methods and optimization-free direct reconstruction methods:
\begin{itemize}
    \item [(i)] The iterative methods formulate the reconstruction procedure as PDE-constrained optimization problems and update the unknown iteratively using gradient descent methods. See, for example, \cite{norton1999, kelley1999, hohage2001, plessix2006, hinze2008}.
    In general, the adjoint-state techniques are used to compute the gradient of the objective functional, which requires solving the forward problem together with the adjoint problem. More recently, to overcome the ill-posedness and the presence of local minima associated with the optimization problem, a recursive linearization method that progressively updates the reconstruction from low to high frequencies was developed \cite{BaoLiLinTriki2015, BorgesGillmanGreengard2017,chen1997}.

    \item [(ii)] The optimization-free methods rely on the construction of imaging functionals or the spectral decomposition of the forward operator to identify the unknown scatterers,  such as the linear sampling method and the factorization method \cite{colton1996, Kirsch1998,  haddar2002, LukePotthast2003, cakoni2006, GriesmaierSchmiedecke2017, CakoniHaddarLechleiter2019}.
    Although computationally attractive, these methods generally recover only geometric information about the imaging target, rather than its quantitative properties.
\end{itemize}

With the rapid development of deep learning, computational methods based on neural networks (NNs) have become an active direction for inverse problems. These data-driven approaches learn prior information about the imaging target from training data and shift the computational burden to the training stage. More specifically, once the NN is trained, the reconstruction can be performed efficiently using the trained network \cite{mccann2017,lucas2018}. Various learning approaches have been developed for solving inverse  problems. For instance, in learning-assisted objective-function approaches, NNs are used to learn certain components of iterative reconstruction within an optimization-based framework, such as descent directions and auxiliary variables \cite{adler2017,guo2019,sanghvi2019}. In
physics-assisted learning approaches, the physical model is incorporated into the learning process,  often through a physics-based approximate inverse that provides a preliminary reconstruction, which is then refined by NNs 
\cite{sun2018,xiao2019,khoshdel2019,wei2019dl,li2019deepnis}.
For inverse scattering problems,
NN architectures have been constructed by using scattering-operator structures, back-propagation procedures, or progressive multi-frequency designs \cite{KhooYing2019, FanYing2022,ZhangZepedaLi2024,Melia2025}. We refer  the reader to \cite{chen2020review} for a more detailed review of these methods and to \cite{Jiang2024, MengZhang2024, Zhang2025} for other 
related developments.

In this work, we propose a multi-level learning framework for inverse scattering problems with multi-frequency measurements. The computational approach combines deep learning with ideas from recursive reconstruction. In this framework,
the multi-level NN architecture is built along the frequency axis of the scattering problem, where a new subnetwork is introduced into the existing computational architecture at each frequency level to update the model (see Figure \ref{fig:architecture} in Section \ref{sc3}). The learning process of the NN is performed via a frequency continuation procedure from low to high frequencies, wherein the newly added subnetwork is trained using the measurement data at each fixed frequency.  There are two main advantages of the proposed method: (i) the multi-level NN is able to recover the higher-order Fourier components of the imaging target as the frequency increases, thereby producing a convergent reconstruction;
(ii) the architectural design yields a robust learning procedure by marching from low to high frequencies.
In particular, the output of the network trained at the previous frequency is incorporated as part of the input to the network at the current frequency. Consequently, the learning problem at each frequency is reduced to a local task, which substantially simplifies the optimization process; see the discussion in Section \ref{sc3}.

To understand the convergence of the proposed computational method, we also provide a theoretical interpretation of the machine learning framework in the neural tangent kernel (NTK) regime~\cite{jac2018}.  Specifically, We show that, at each frequency level, the Fourier transform of the network output agrees with that of the imaging target on a frequency set determined jointly by the current-frequency data and the Fourier information inherited from the previous frequency levels.
Therefore, the computational method leads to a more accurate reconstruction as the frequency increases.

The proposed NN architecture shares a similar spirit with the multi-grade learning framework \cite{xu2025multi} and other multi-stage learning strategies, where a complicated learning task is decomposed into simpler local learning tasks \cite{bengio2007greedy,belilovsky2019greedy}. Here we employ a robust learning strategy based on the frequency-continuation idea for the recursive linearization methods
\cite{BaoLiLinTriki2015}. Such a training strategy is robust because the training at each frequency level can be carried out in a reliable and efficient manner, and by marching along the frequencies, the overall procedure reduces the risk of being trapped in local minima when solving the optimization problem. Furthermore, as discussed above, it is promising to obtain a convergent reconstruction by recovering the higher-order Fourier modes of the imaging target as the frequency increases.

The remainder of this paper is organized as follows. In Section \ref{sc2}, we introduce the inverse scattering problems and focus on two representative settings: inverse source scattering and inverse medium scattering. Section \ref{sc3} presents the multi-level machine learning computational framework, where we describe the NN architecture and the training algorithm. Various numerical examples are shown in Section \ref{sc4} to illustrate the effectiveness and robustness of the proposed framework. In Section \ref{sc5}, we present the theoretical analysis of learning dynamics in the NTK regime. Finally a conclusion is given in Section \ref{sc6}. 

\section{Problem Setting} \label{sc2}
We consider two representative inverse scattering problems: the inverse source problem (ISP) and the inverse medium problem (IMP). 
The former seeks to recover
a source function from the radiated field, while the latter aims to recover the
 medium property of an object from the scattered field.

\medskip
\noindent\textbf{ISP.} Let
\(
\Omega=\{\mathbf r\in\mathbb R^2:\ |\mathbf r|<R\}
\), and \(f^\dagger\) be a real-valued source function compactly supported in \(\Omega\).  For the wavenumber \(\kappa>0\), the radiated field \(\psi(\kappa,\mathbf r)\) generated by \(f^\dagger\) satisfies
\begin{equation}\label{1}
\Delta \psi+\kappa^2\psi=-f^\dagger
\qquad \text{in } \mathbb{R}^2,
\end{equation}
together with the Sommerfeld radiation condition
\[
\lim_{|\mathbf r|\to\infty} |\mathbf r|^{1/2}
\left(
\frac{\partial \psi}{\partial |\mathbf r|}-i\kappa\psi
\right)=0.
\]
The ISP is to recover \(f^\dagger\) from the boundary measurement data of the radiated field \(\psi\) on \(\partial\Omega\). For a fixed wavenumber \(\kappa\), the data are denoted by \(D_\kappa\).

\medskip
\noindent\textbf{IMP.} Let \(q^\dagger>-1\) be a real-valued inhomogeneity of the medium compactly supported in \(\Omega\). Given the wavenumber \(\kappa>0\), the total field \(\phi^{\mathrm{total}}(\kappa,\mathbf r)\) satisfies
\begin{equation}
\Delta \phi^{\mathrm{total}}+\kappa^2(1+q^\dagger)\phi^{\mathrm{total}}=0
\qquad \text{in } \mathbb{R}^2.
\end{equation}
 Let
\(
\phi^{\mathrm{inc}}(\kappa, \mathbf r)=e^{i\kappa \mathbf p\cdot \mathbf r}
\)
be an incident plane wave, where \(\mathbf p=(\cos\beta,\sin\beta)\in\mathbb S^1\) for a given \(\beta\in[0,2\pi)\).
We decompose the total field as
\[
\phi^{\mathrm{total}}=\phi^{\mathrm{inc}}+\psi,
\]
where \(\psi(\kappa,\mathbf r)\) is the scattered field, which satisfies
\begin{equation}\label{3}
\Delta \psi+\kappa^2(1+q^\dagger)\psi
=
-\kappa^2 q^\dagger \phi^{\mathrm{inc}}
\qquad \text{in } \mathbb{R}^2,
\end{equation}
together with the Sommerfeld radiation condition
\[
\lim_{|\mathbf r|\to\infty} |\mathbf r|^{1/2}
\left(
\frac{\partial \psi}{\partial |\mathbf r|}-i\kappa\psi
\right)=0.
\]
The IMP aims to recover \(q^\dagger\) from the boundary measurements of the scattered field \(\psi\) on \(\partial\Omega\). For a given wavenumber \(\kappa\) and incident direction \(\mathbf p\), the corresponding data are denoted by \(D_{\kappa,\mathbf p}\). For notational simplicity, we also write \(D_\kappa\) when \(\mathbf p\) is
fixed or clear from the context.

\medskip
 \noindent\textbf{A unified formulation with multi-frequency data.} 
We formulate the two inverse problems above in a unified framework. To this
end, let $\mathcal U$ be the space of admissible functions, and let $a^\dagger\in\mathcal U$ be the quantity to be
reconstructed. For a given wavenumber $\kappa>0$, let $\psi$ denote the
solution of the corresponding forward problem associated with $a^\dagger$.
Namely, in the ISP, $a^\dagger$ denotes the source term $f^\dagger$, and
$\psi$ solves \eqref{1}; in the IMP, $a^\dagger$ denotes the medium
inhomogeneity $q^\dagger$, and $\psi$ is the scattered field solving \eqref{3} for the prescribed incident
field.

Given a measurement space $\mathcal Y$, we define the forward operator
\[
    \mathscr F_\kappa:\mathcal U\to\mathcal Y
\]
through the corresponding forward model, that is \eqref{1} for the ISP and
\eqref{3} for the IMP. It maps the unknown \(a^\dagger\) to the measurement data \(D_\kappa\)
associated with \(\psi\) at \(\kappa\), namely,
\begin{equation}\label{4}
    D_\kappa
    =
    \mathscr F_\kappa(a^\dagger).
\end{equation}

We consider multi-frequency inverse scattering problems, in which the
corresponding measurements are given at a sequence of wavenumbers
\[
    0<\kappa_1<\kappa_2<\cdots<\kappa_N.
\]
Then
\begin{equation}\label{5}
     D_{\kappa_n}
     =
     \mathscr F_{\kappa_n}(a^\dagger),
     \qquad n=1,\ldots,N,
\end{equation}
are the measurement data at wavenumber $\kappa_n$. When noise is present, we write
\[
    D_{\kappa_n}^\delta
    =
    \mathscr F_{\kappa_n}(a^\dagger)+\epsilon_{\kappa_n},
    \qquad n=1,\ldots,N,
\]
where $\epsilon_{\kappa_n}$ denotes the measurement error. The goal is to
recover $a^\dagger$ from the multi-frequency data
\(
    \{D_{\kappa_n}\}_{n=1}^N
\)
or the noisy data
\(
    \{D_{\kappa_n}^\delta\}_{n=1}^N.
\)

In the settings of ISP and IMP considered in this work, the measurement data
\(D_\kappa\) are taken to be the Cauchy boundary data, namely,
\begin{equation}\label{eq:cauchy_data}
D_\kappa
=
\left(
\psi|_{\partial\Omega},
\partial_{\mathbf n}\psi|_{\partial\Omega}
\right),
\end{equation}
where \(\mathbf n\) is the unit outward normal vector on \(\partial\Omega\), and
\(\partial_{\mathbf n}\psi=\nabla\psi\cdot \mathbf n\) denotes the normal derivative on \(\partial\Omega\).

\section{Multi-Level Learning Framework}\label{sc3}

In this section, we present the multi-level learning framework for multi-frequency inverse scattering problems. The central idea is to organize the learning process along an increasing sequence of frequencies, with each
frequency represented by its corresponding wavenumber \(\kappa\). At each frequency level, a subnetwork is introduced to update the current approximation by taking two inputs:  the output from the previous level, which
incorporates information from all preceding frequencies, and a domain-defined feature constructed from the current-frequency data through the underlying physical model. As such, the learning task is reduced to successive local learning problems between consecutive frequency levels, improving the stability and efficiency of the training process.

\subsection{The Multi-Level Neural Network Architecture}
A direct end-to-end mapping from all frequency measurements to the imaging target is difficult to learn, since the corresponding objective functional may contain many undesirable local minima. To overcome this challenge and make effective use of the multi-frequency data, we propose a learning process that proceeds along the frequency axis, decomposing the global learning task into a sequence of simpler local learning tasks, each associated with a subnetwork corresponding to one frequency level.
\begin{figure}[htbp]
\centering
\resizebox{0.86\textwidth}{!}{%
\begin{tikzpicture}[
basebox/.style={
    draw,
    rounded corners=1.5mm,
    align=center,
    thick,
    font=\small,
    inner sep=4pt,
    text=black!90
},
redbox/.style={
    basebox,
    fill=red!9,
    draw=red!42!black,
    minimum width=1.35cm,
    minimum height=0.65cm
},
bluebox/.style={
    basebox,
    fill=blue!8,
    draw=blue!42!black,
    minimum width=2.5cm,
    minimum height=1.05cm
},
greenbox/.style={
    basebox,
    fill=green!9,
    draw=green!42!black,
    minimum width=1.5cm,
    minimum height=0.58cm
},
targetbox/.style={
    draw=orange!42!black,
    fill=orange!10,
    ellipse,
    minimum width=2.2cm,
    minimum height=0.75cm,
    thick,
    font=\small,
    text=black!90
},
levelbox/.style={
    draw=gray!48,
    rounded corners=3mm,
    thick,
    inner sep=8pt
},
arr/.style={
    -{Stealth[scale=1.0]},
    thick,
    draw=black!68
},
skiparr/.style={
    -{Stealth[scale=1.0]},
    thick,
    draw=blue!42!black,
    rounded corners=2mm
},
dasharr/.style={
    -{Stealth[scale=0.9]},
    dashed,
    draw=gray!62
}
]

% Level 1
\node[bluebox] (b1) {$\mathcal{N}_{\Theta_{1}}(\mathbf{x}_{1})$};
\node[redbox, above=0.48cm of b1] (r1) {$\tilde{a}_1$};
\node[greenbox, below=0.48cm of b1] (g1) {$a_1$};

% Initial input outside Level 1 block
\node[left=0.75cm of b1, font=\small, text=black!90] (a0) {$a_0$};
\node[above=0.03cm of a0, font=\scriptsize, text=gray!75] {Initial};

\node[levelbox, fit=(r1) (b1) (g1)] (box1) {};
\node[above=0.12cm of box1, font=\footnotesize, text=black!80] {$\kappa_1$};

% Level 2
\node[bluebox, right=1.35cm of b1] (b2) {$\mathcal{N}_{\Theta_{2}}(\mathbf{x}_{2})$};
\node[redbox, above=0.48cm of b2] (r2) {$\tilde{a}_2$};
\node[greenbox, below=0.48cm of b2] (g2) {$a_2$};
\node[levelbox, fit=(r2) (b2) (g2)] (box2) {};
\node[above=0.12cm of box2, font=\footnotesize, text=black!80] {$\kappa_2$};

% Dots
\node[right=0.75cm of b2, font=\Large, text=gray!70] (dots) {$\cdots$};

% Level N-1
\node[bluebox, right=0.75cm of dots] (bn1) {$\mathcal{N}_{\Theta_{N-1}}(\mathbf{x}_{N-1})$};
\node[redbox, above=0.48cm of bn1] (rn1) {$\tilde{a}_{N-1}$};
\node[greenbox, below=0.48cm of bn1] (gn1) {$a_{N-1}$};
\node[levelbox, fit=(rn1) (bn1) (gn1)] (boxn1) {};
\node[above=0.12cm of boxn1, font=\footnotesize, text=black!80] {$\kappa_{N-1}$};

% Level N
\node[bluebox, right=1.35cm of bn1] (bn) {$\mathcal{N}_{\Theta_N}(\mathbf{x}_N)$};
\node[redbox, above=0.48cm of bn] (rn) {$\tilde{a}_N$};
\node[greenbox, below=0.48cm of bn] (gn) {$a_N$};
\node[levelbox, fit=(rn) (bn) (gn)] (boxn) {};
\node[above=0.12cm of boxn, font=\footnotesize, text=black!80] {$\kappa_N$};

% Vertical arrows
\foreach \i in {1,2,n1,n} {
    \draw[arr] (r\i) -- (b\i);
    \draw[arr] (b\i) -- (g\i);
}

% Initial input arrow
\draw[skiparr] (a0.east) -- (b1.west);

% Cross-level transfer
\draw[skiparr] (g1.east) -- ++(0.65,0) |- (b2.west);
\draw[skiparr] (g2.east) -- ++(0.65,0) |- (dots.west);
\draw[skiparr] (dots.east) -- (bn1.west);
\draw[skiparr] (gn1.east) -- ++(0.65,0) |- (bn.west);

% Target and supervision
\path (box2.south) -- (boxn1.south) coordinate[midway] (midbottom);
\node[targetbox, below=1.6cm of midbottom] (target) {Target $a^\dagger$};

\draw[dasharr] (g1.south) -- (target.165);
\draw[dasharr] (g2.south) -- (target.140);
\draw[dasharr] (gn1.south) -- (target.40);
\draw[dasharr] (gn.south) -- (target.15);

% Top arrow
\coordinate (topleft) at ([yshift=0.9cm, xshift=-0.2cm]box1.north west);
\coordinate (topright) at ([yshift=0.9cm, xshift=0.2cm]boxn.north east);

\draw[-{Stealth[scale=1.1]}, line width=1.2pt, draw=gray!60]
    (topleft) -- (topright)
    node[midway, above=0.12cm, text=black!90, font=\small]
    {\(\kappa\)};

\end{tikzpicture}
}
\caption{Schematic illustration of the proposed multi-level machine learning framework. At each frequency level, the corresponding subnetwork uses the previous output and a transformed feature to update the approximation, which is then propagated
to the next level. }
\label{fig:architecture}
\end{figure}
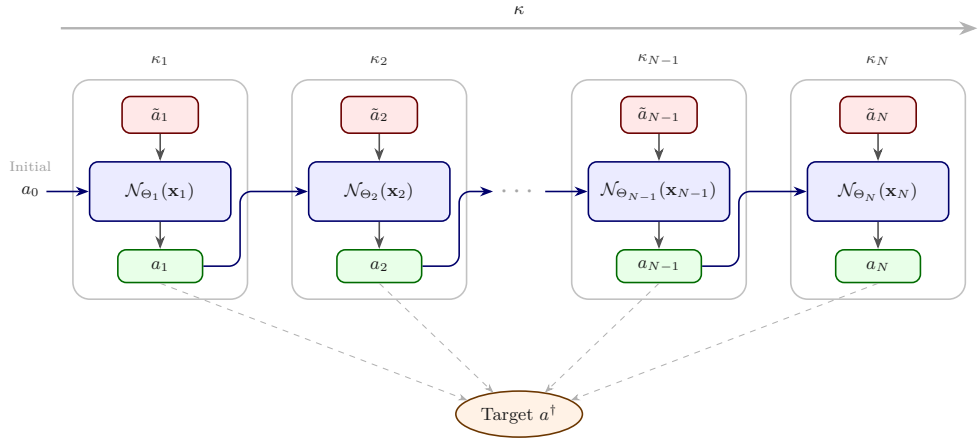

Assume that the measurement data \(\{D_{\kappa_n} \}_{n=1}^N\) are available at $\{\kappa_n\}_{n=1}^N$. 
The $N$-level architecture constructed in this work is illustrated schematically in Figure~\ref{fig:architecture}. 
In this architecture, a subnetwork 
\(\mathcal N_{\Theta_n}\) is assigned to the \(n\)-th frequency level,
corresponding to the wavenumber \(\kappa_n\), where \(\Theta_n\) denotes the
trainable parameters of the subnetwork. Starting with an initial guess \(a_0\), the NN model is updated
recursively along the frequency axis, beginning at the lowest frequency corresponding to \(\kappa_1\).
 At the \(n\)-th frequency level, the subnetwork
\(\mathcal N_{\Theta_n}\) updates the approximation of \(a^\dagger\) by
combining the previous-level output  \(a_{n-1}\), which carries
the information accumulated from lower frequencies, with a feature \(\tilde a_n\) extracted from the current
measurement data \(D_{\kappa_n}\). More precisely, for \(n=1,2,\ldots,N\), the input \(\mathbf x_n\) to
\(\mathcal N_{\Theta_n}\) is given by
\[
    \mathbf x_n := (a_{n-1},\tilde a_n).
\]
The updated approximation \(a_n\) is then given by
\begin{equation}\label{7}
   a_n := \mathcal N_{\Theta_n}(\mathbf x_n),
   \qquad n=1,2,\ldots,N. 
\end{equation}
Each subnetwork \(\mathcal N_{\Theta_n}\) is trained so that its output \(a_n\)
approximates the target \(a^\dagger\) associated with the inverse problem
\eqref{5}. As the frequency increases from \(\kappa_{n-1}\) to \(\kappa_n\), the
newly added subnetwork \(\mathcal N_{\Theta_n}\) updates the previous
approximation \(a_{n-1}\) by incorporating the measurement data at
\(\kappa_n\). The resulting approximation \(a_n\) is expected to provide a
more accurate reconstruction than \(a_{n-1}\). Therefore, by marching from the
lowest frequency, equivalently \(\kappa_1\), to the highest frequency, equivalently \(\kappa_N\), the proposed architecture progressively builds a refined model
for solving the inverse problem.

With this architecture, it remains to specify how the data are used to construct the feature \(\widetilde a_n\) for the \(n\)-th subnetwork $N_{\Theta_n}$. If the measurement data is directly used as the network input, each network would have to learn both the transformation from the measurement space to the domain of the target \(a^\dagger\) and the reconstruction update. To avoid this additional computational burden, we exploit the mathematical structure of the underlying forward model to convert the measurement data \(D_{\kappa_n}\) into a domain-defined feature \(\widetilde a_n\), which is compatible with the output \(a_{n-1}\) from the previous frequency level.

We describe this data-to-domain transformation in a general form as follows.
For a fixed \(\kappa>0\), let \(D_\kappa\) be the measurement data defined in \eqref{4}. Using the underlying physical model, the measurement data and the target are related through an approximate relation of the form
\begin{equation}\label{8}
\mathscr J_{\kappa}(D_{\kappa})
\approx
\mathscr L_{\kappa}(a^\dagger),
\end{equation} 
where \(\mathscr J_\kappa\) is a transformation constructed from the measurement data \(D_\kappa\) and
\(\mathscr L_\kappa\) is an explicitly defined operator acting on the target \(a^\dagger\). The specific forms of \(\mathscr J_{\kappa}\) and \(\mathscr L_{\kappa}\) for the two inverse scattering problems under consideration will be given below. Moreover, a domain-defined feature \(\widetilde a_\kappa\) is constructed as
\begin{equation}\label{9}
    \tilde a_\kappa
    :=
    \mathscr R_\kappa \mathscr J_\kappa(D_\kappa),
\end{equation}
where \(\mathscr R_\kappa\) denotes an inverse-type transformation of
of \(\mathscr L_\kappa\).  The quantity \(\widetilde a_\kappa\) serves as some approximation of \(a^\dagger\) constructed from \(D_\kappa\). At the \(n\)-th frequency level, we set
\[
    \tilde a_n
    :=
    \tilde a_{\kappa_n}
    =
    \mathscr R_{\kappa_n}\mathscr J_{\kappa_n}(D_{\kappa_n}).
\] 
We now apply this construction, using the Cauchy boundary data defined in \eqref{eq:cauchy_data}, to the two model problems considered in this work: the ISP and the IMP.

\medskip
\noindent\textbf{ISP.} For a given $\theta\in[0,2\pi)$, we consider an auxiliary function
\begin{equation}\label{10}
\phi(\kappa,\mathbf r)=e^{\mathrm{i}\kappa \mathbf d\cdot \mathbf r}
,
\qquad
\mathbf d=(\cos\theta,\sin\theta)\in\mathbb S^1.  
\end{equation}
Multiplying \eqref{1} by \(\phi\) and applying Green's identity over \(\Omega\) gives
\begin{equation}\label{11}
\int_{\partial\Omega}
\left(
\psi\,\partial_{\mathbf n}\phi
-
\phi\,\partial_{\mathbf n}\psi
\right)\, \mathrm d s
=
\int_{\Omega} f^\dagger\,\phi\,\mathrm d\mathbf r.
\end{equation}
Accordingly, we define 
\[
\mathscr J_\kappa(D_\kappa):= \int_{\partial\Omega}\left(
\psi\,\partial_{\mathbf n}\phi
-
\phi\,\partial_{\mathbf n}\psi
\right)\, \mathrm d s
\quad \mbox{and}
\quad
\mathscr L_\kappa(f^\dagger):= \int_{\Omega} f^\dagger\,\phi\,\mathrm d\mathbf r.
\]
The corresponding inverse transformation \(\mathscr R_\kappa\) is then applied to construct the domain-defined feature  \(\tilde f_\kappa\) according to \eqref{9}.

\medskip
\noindent\textbf{IMP.}
Based on \eqref{3}, let \(\psi_B(\kappa, \mathbf r)\) denote the Born approximation of the scattered field, satisfying
\begin{equation}\label{12}
\Delta\psi_B+\kappa^2\psi_B
=
-\kappa^2 q^\dagger\,\phi^{\mathrm{inc}}
\qquad \text{in } \Omega.
\end{equation}
Recall that
\(
\phi^{\mathrm{inc}}(\kappa, \mathbf r)
=
e^{\mathrm{i}\kappa \mathbf p\cdot \mathbf r}
\).
Multiplying \eqref{12} by \(\phi\) defined as in \eqref{10} and applying Green's identity over \(\Omega\), we obtain
\begin{equation}\label{13}
\kappa^{-2}
\int_{\partial\Omega}
\left(
\psi_B\,\partial_{\mathbf n}\phi
-
\phi\,\partial_{\mathbf n}\psi_B
\right)\,\mathrm ds
=
\int_{\Omega} q^\dagger\,\phi\,\phi^{\mathrm{inc}}\,\mathrm d\mathbf r.
\end{equation}
Since
\[
\int_{\Omega} q^\dagger\,\phi\,\phi^{\mathrm{inc}}\, \mathrm d\mathbf r
=
\int_\Omega q^\dagger(\mathbf r)
e^{i\kappa(\mathbf d+\mathbf p)\cdot \mathbf r}\, \mathrm d\mathbf r,
\]
the right-hand side of \eqref{13} gives Fourier-type data of \(q^\dagger\) at
\[
\boldsymbol{\xi} = \kappa(\mathbf d+\mathbf p).
\]
In the notation of \eqref{8}, we define
\[
\mathscr L_\kappa(q^\dagger)
:=
\int_\Omega q^\dagger(\mathbf r)
e^{\mathrm{i}\boldsymbol{\xi}\cdot \mathbf r}\,\mathrm d\mathbf r.
\]
In the construction, we replace the Born field \(\psi_B\) in \eqref{13} with the measured scattered field \(\psi\), and define
\[
\mathscr J_\kappa(D_\kappa)
:=
\kappa^{-2}
\int_{\partial\Omega}
\left(
\psi\,\partial_{\mathbf n}\phi
-
\phi\,\partial_{\mathbf n}\psi
\right)\,\mathrm ds.
\]
Accordingly, \eqref{8} takes the form
\begin{equation}\label{14}
  \kappa^{-2}
\int_{\partial\Omega}
\left(
\psi\,\partial_{\mathbf n}\phi
-
\phi\,\partial_{\mathbf n}\psi
\right)\,\mathrm ds
\approx
\int_\Omega q^\dagger(\mathbf r)
e^{\mathrm{i}\boldsymbol{\xi}\cdot \mathbf r}
\,\mathrm d\mathbf r.  
\end{equation}
Applying \(\mathscr R_\kappa\) as in \eqref{9} then yields the corresponding domain-defined feature \(\tilde q_\kappa\).

\subsection{Multi-Level Training Algorithm}\label{sc3.2}

The proposed framework is trained level by level from low to high frequencies.  At each frequency level, the subnetwork takes a two-channel input consisting of
the previous-level output and the current-frequency feature.
Specifically, for the \(i\)-th training sample, the input to the \(n\)-th subnetwork is formed as
\[
\mathbf x_{n,[i]}
=
\bigl(a_{n-1,[i]},\widetilde a_{n,[i]}\bigr),
\qquad i=1,\ldots,M,
\]
where $a_{n-1,[i]}$ denotes the output of the previous subnetwork for the $i$-th sample, with $a_{0,[i]}$ set to the initial guess, and \(\widetilde a_{n,[i]}\) denotes the domain-defined feature constructed from
the current-frequency data according to \eqref{8} and \eqref{9}.
The parameters of the $n$-th subnetwork are obtained by solving
\[
\Theta_n
=
\arg\min_{\Theta}
\frac{1}{M}
\sum_{i=1}^{M}
\mathcal L\left(
\mathcal N_{\Theta}(\mathbf x_{n,[i]}), a_{[i]}^\dagger
\right),
\]
where \(\mathcal L\) denotes the loss function between the network output and the
reference target. Since all previously trained subnetworks are fixed, the optimization at level \(n\) only involves the parameters of the current network. After training, the $n$-th subnetwork produces the updated approximation
\[
a_{n,[i]} = \mathcal{N}_{\Theta_n}(\mathbf{x}_{n,[i]}),
\]
which is then passed to the next frequency level.

We summarize the complete computational procedure in Algorithm~\ref{algo}. As the frequency increases, the higher-order Fourier components of the imaging target are progressively recovered, which improves the accuracy of reconstruction. In addition, by decomposing the global learning problem into a sequence of simpler local tasks, the computational framework  reduces the influence of undesirable local minima and improves the stability of the overall iterative reconstruction process.

\begin{algorithm}[H]
\caption{Multi-level training with multi-frequency measurements at
\(\{\kappa_n\}_{n=1}^N\)}
\label{algo}
\begin{algorithmic}[1]
\Require Training samples 
\(\{(a^\dagger_{[i]},\{\widetilde a_{n,[i]}\}_{n=1}^N)\}_{i=1}^M\), 
initial guesses \(\{a_{0,[i]}\}_{i=1}^M\)
\For{\(n=1,\ldots,N\)}
    \For{\(i=1,\ldots,M\)}
        \State Form the input
        \[
        \mathbf{x}_{n,[i]} = \bigl(a_{n-1,[i]},\,\tilde{a}_{n,[i]}\bigr)
        \]
    \EndFor
    \State Train the \(n\)-th network by solving
    \[
    \Theta_n
    =
    \arg\min_{\Theta}
    \frac{1}{M}\sum_{i=1}^M
    \mathcal L\Bigl(
     \mathcal{N}_{\Theta}(\mathbf{x}_{n,[i]}),\,
        a^\dagger_{[i]}
    \Bigr)
    \]
    \For{\(i=1,\ldots,M\)}
        \State Update the approximation
        \[
       a_{n,[i]} = \mathcal{N}_{\Theta_n}(\mathbf{x}_{n,[i]})
        \]
    \EndFor
\EndFor
\State \textbf{Output:} \(\{\mathcal{N}_{\Theta_n}\}_{n=1}^N\)
\end{algorithmic}
\end{algorithm}

\section{Numerical Experiments}\label{sc4}

In this section, we present several numerical experiments for the ISP and the IMP to demonstrate the performance of the proposed multi-level learning framework.

We parameterize the wavenumber by \(\kappa=2\pi\nu\), where \(\nu\) denotes the frequency. In the numerical experiments, we take \(0<\nu_{\min}<\nu_{\max}\) and \(N\geq 2\). The frequencies in \([\nu_{\min},\nu_{\max}]\) are uniformly sampled as
\[
  \nu_n=\nu_{\min}+(n-1)\delta\nu,
  \qquad
  \delta\nu=\frac{\nu_{\max}-\nu_{\min}}{N-1},
  \qquad n=1,\ldots,N.
\]
The corresponding wavenumbers are
\[
  \kappa_n=2\pi\nu_n,
  \qquad n=1,\ldots,N .
\]
In all examples, we let
\[
\Omega=\{\mathbf r\in\mathbb R^2:|\mathbf r|<1.5\},
\]
and assume that the unknown source or medium is compactly supported in \(\Omega\). 
To generate the training data, we solve the forward problems using an
\(H^1\)-conforming complex-valued finite element method implemented in NGSolve. To approximate the outgoing radiation condition, the open domain is truncated
to a finite computational domain using the perfectly matched layer (PML)
technique.

For each \(\kappa\in\{\kappa_n\}_{n=1}^N\), the measurement data
\(D_\kappa\) are given by the Cauchy boundary data of the radiated or scattered
field collected on \(\partial\Omega\), as defined in
\eqref{eq:cauchy_data}.

The data are contaminated by noise. Specifically, for a noise level \(\delta> 0\),  the
noisy Cauchy data \(D_\kappa^\delta=\left(
\psi^\delta|_{\partial\Omega},
\partial^\delta_{\mathbf n}\psi|_{\partial\Omega}
\right)\) are generated by 
\[
\psi^\delta|_{\partial\Omega}
=
\psi|_{\partial\Omega}+\epsilon_\kappa^{(1),\delta},
\qquad
\partial_{\mathbf n}\psi^\delta|_{\partial\Omega}
=
\partial_{\mathbf n}\psi|_{\partial\Omega}+\epsilon_\kappa^{(2),\delta} .
\]
Here, \(\epsilon_\kappa^{(1),\delta}\) and \(\epsilon_\kappa^{(2),\delta}\) are independent complex-valued uniform random perturbations added to the discretized Dirichlet and Neumann data, respectively. They are scaled separately so that
\[
\frac{\|\epsilon_\kappa^{(1),\delta}\|_2}
{\|\psi|_{\partial\Omega}\|_2}
=
\delta,
\qquad
\frac{\|\epsilon_\kappa^{(2),\delta}\|_2}
{\|\partial_{\mathbf n}\psi|_{\partial\Omega}\|_2}
=
\delta,
\]
where \(\|\cdot\|_2\) denotes the discrete \(\ell^2\)-norm.

\begin{figure}[H]
  \centering
\includegraphics[width=0.80\linewidth]{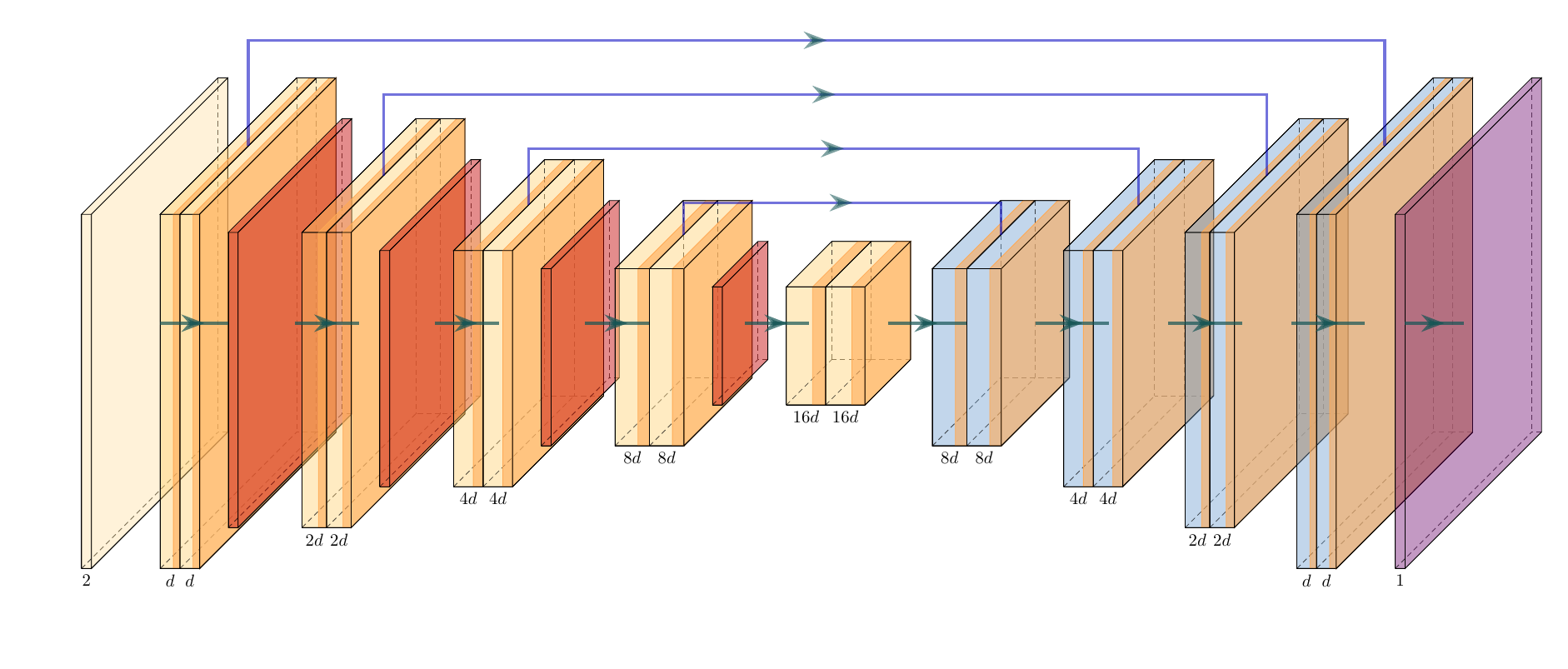}
\caption{U-Net subnetwork with an encoder-decoder architecture and skip connections. The network takes a two-channel input and produces a single-channel output. The encoder has four downsampling stages with \(d\), \(2d\), \(4d\), and \(8d\) feature channels, followed by a \(16d\)-channel bottleneck, where \(d\) is a positive integer. Each encoder stage contains two \(3\times3\) convolutional layers with ReLU activations, followed by \(2\times2\) average pooling. The decoder is symmetric and uses nearest-neighbor upsampling and encoder--decoder skip connections. Each decoding stage applies two \(3\times3\) convolutional layers with ReLU activations after feature concatenation. A final \(1\times1\) convolution produces the single-channel output.}
  \label{fig:u-net}
\end{figure}

In Algorithm~\ref{algo}, the initial guess \(a_{0,[i]}\) for each sample is taken to be the zero function. For each \(n\), the input to the subnetwork, as defined in \eqref{7}, consists of matrix-valued quantities defined on the same grid as the ground-truth target. Therefore, NN techniques for image processing can be
directly employed for training.
In particular, the subnetwork is implemented as a convolutional U-Net, as illustrated in Figure~\ref{fig:u-net}. The U-Net takes a two-channel input of spatial resolution \(128\times128\) and has a single-channel output. When applied to the previous-level reconstruction \(a_{n-1}\) and the current-frequency feature \(\widetilde a_n\), the trained subnetwork \(\mathcal N_{\Theta_n}\) yields the updated reconstruction \(a_n\).
 The U-Net has an encoder-decoder architecture with skip connections.
The encoder consists of four downsampling blocks. Each downsampling block consists of two \(3\times3\) convolutional layers, each followed by a ReLU activation, before a \(2\times2\) average-pooling layer is applied. The bottleneck consists of two \(3\times3\) convolutional layers with ReLU activations. The decoder consists of four upsampling blocks. At each decoding level, the feature map is first upsampled by a factor of \(2\) using nearest-neighbor interpolation and then concatenated, along the channel dimension, with the pre-pooling feature map from the corresponding encoder block through a skip connection. The concatenated feature map is subsequently processed by two \(3\times3\) convolutional layers with ReLU activations. Finally, a \(1\times1\) convolution reduces the resulting feature map to a single channel.  The source code is available at
\href{https://github.com/liuyi3616/multi-frequency-inverse-problems}
{\texttt{GitHub}}.

For a reconstruction \(a_n^\delta\) obtained at the \(n\)-th frequency level
from the data with noise level \(\delta\), we define the relative error by
\[
e_n^\delta
=
\frac{\|a_n^\delta-a^\dagger\|_F}{\|a^\dagger\|_F}.
\]
In each reconstruction, the relative error is given in the upper-right corner of the image.

\subsection{Inverse Source Problem}
We set \(\nu_{\min}=0.1\), \(\nu_{\max}=7\) and the frequency step size \(\delta\nu=0.1\). 
The number of frequency levels is \(N=70\).
In the following numerical experiments, we consider data with two noise levels: $\delta=0$ and $\delta=0.05$.

\medskip
\noindent\textbf{Data generation and training setup.} 
We generate \(6000\) synthetic source functions. Each source \(f^{\dagger}\) consists of \(1\)--\(5\) non-overlapping inclusions, including cosine disks, rectangles, and triangles, with support compactly contained in \(\Omega\). The rectangular and triangular inclusions are slightly smoothed, and a smooth radial cutoff is applied so that the source vanishes near \(\partial\Omega\).

For each \(n=1,\ldots,N\), the Cauchy data are defined as
\[
    D_{\kappa_n}
    =
    \left(
    \psi_n|_{\partial\Omega},
    \partial_{\mathbf n}\psi_n|_{\partial\Omega}
    \right),
\]
where \(\psi_n\) is the radiated field satisfying
\[
\Delta \psi_n+\kappa_n^2\psi_n=-f^\dagger \qquad \text{in } \mathbb R^2. 
\]
In the numerical implementation, \(D_{\kappa_n}\) is approximated by solving the corresponding PML-truncated forward problem.

To construct the domain feature, according to  \eqref{10} and \eqref{11}, we first choose \(N_\theta=30\) uniformly sampled directions on the unit circle:
\[
\theta_j=\frac{2\pi(j-1)}{N_\theta},
\qquad
\mathbf d_j=(\cos\theta_j,\sin\theta_j),
\qquad j=1,\ldots,N_\theta .
\]
For the \(j\)-th direction at the \(n\)-th frequency level, we define
\[
\phi_{n,j}(\mathbf r)
=
e^{\mathrm i\kappa_n \mathbf d_j\cdot \mathbf r}.
\]
Using the Cauchy data \(D_{\kappa_n}\), we compute the coefficient
\begin{equation}\label{bound_int}
\widehat f_{n}(\theta_j)
=
\int_{\partial\Omega}
\left(
\psi_n\,\partial_{\mathbf n}\phi_{n,j}
-
\phi_{n,j}\,\partial_{\mathbf n}\psi_n
\right)\,\mathrm ds .
\end{equation}
Then, the domain feature  \(\widetilde f_n\) at the \(n\)-th level is given by
\begin{equation}\label{16}
\tilde f_{n}(\mathbf r)
=
\sum_{j=1}^{N_\theta}
\widehat f_{n}(\theta_j)
e^{-\mathrm i\kappa_n\mathbf d_j\cdot \mathbf r}.   
\end{equation}

For the training procedure in Algorithm \ref{algo}, each ground-truth source function  \(f_{[i]}^\dagger\), and its corresponding domain-defined features \(\{\widetilde f_{n,[i]}\}_{n=1}^N\) are discretized on a \(128\times128\) grid and stored as images. The resulting dataset consists of \(6000\) synthetic samples, each containing one ground-truth source and its corresponding multi-frequency features. The dataset is partitioned into training, validation, and test sets of sizes \(3840\), \(960\), and \(1200\), respectively.
We use a U-Net architecture with a base number of feature channels $d=8$ for each subnetwork
and train the subnetworks using \textsc{Adam}
with an initial learning rate of $10^{-4}$, a batch size of $80$, and $400$ epochs.
An exponential learning-rate scheduler with decay factor $\chi=0.99$ is applied.

\medskip
\noindent\textbf{Numerical results.} The numerical results below are organized in three parts. First, we demonstrate
the reconstruction quality of the proposed computational method. Second, we
compare the proposed method with the Landweber iterative method. Third, we test the trained model on samples outside the training distribution.

\begin{figure}[H]
  \centering
  \includegraphics[width=0.94\linewidth]{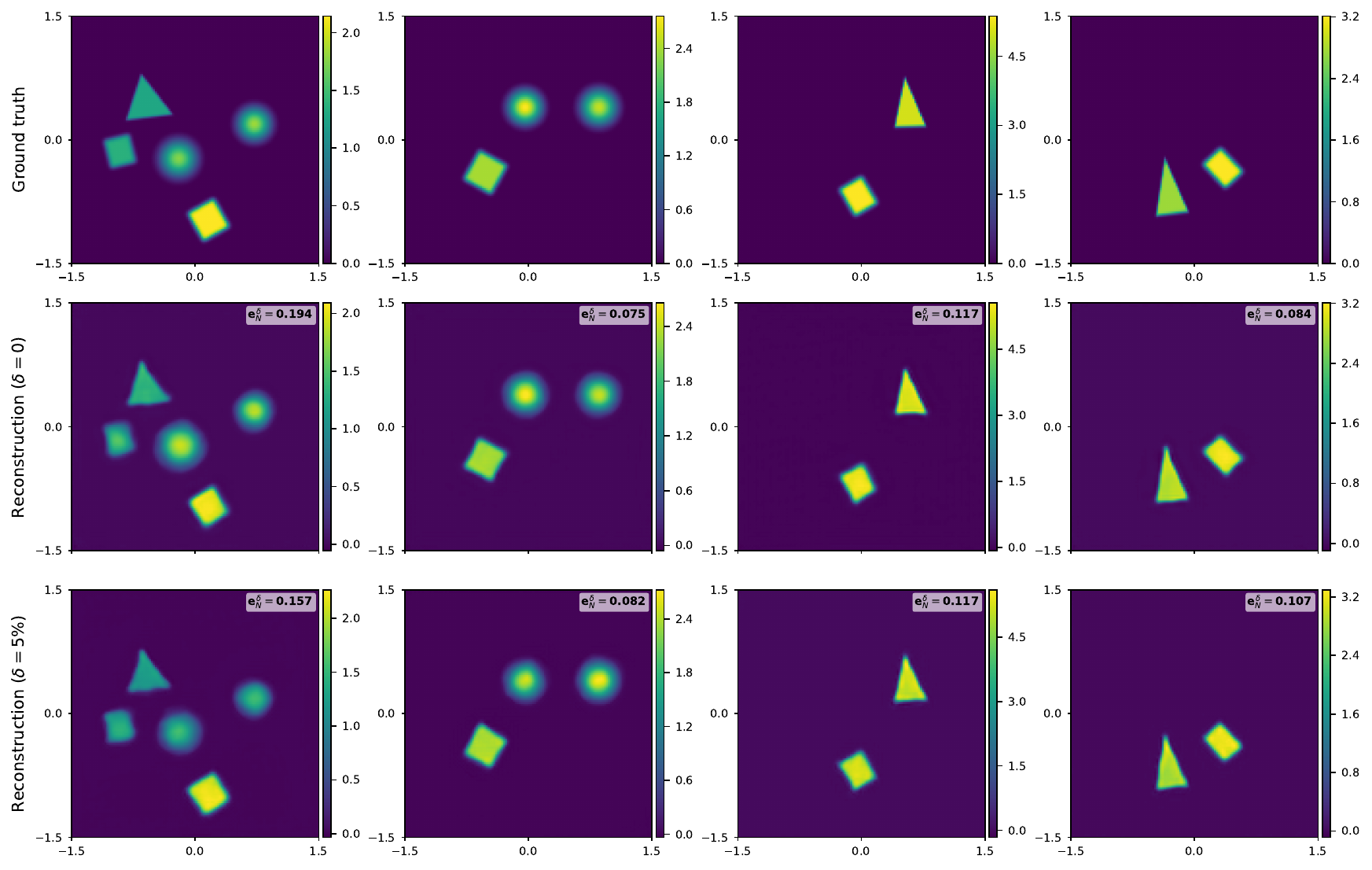}
\caption{Final reconstructions for the ISP with \(\nu_N=\nu_{\max}\).
Top row: ground truth \(f^\dagger\).
Middle row: reconstructions \(f_N^\delta\) with \(\delta=0\).
Bottom row: reconstructions \(f_N^\delta\) with \(\delta=0.05\).}
  \label{fig1:f7-compare}
\end{figure}

\smallskip
\noindent\textbf{(i) Reconstruction results.}
Figure~\ref{fig1:f7-compare} presents four representative test samples and compares
the ground truth \(f^\dagger\) with the final reconstructions \(f_N^\delta\)
at \(\nu_N=\nu_{\max}\) obtained from models trained at different noise levels.
In each column, the top row shows the ground-truth source function, the middle
row shows the reconstruction from the model trained with \(\delta=0\), and the
bottom row shows the reconstruction from the model trained with \(\delta=0.05\).
The reconstructions accurately recover the locations and shapes of the radiating
sources. In addition, the relative errors for \(\delta=0.05\) remain comparable
to those for \(\delta=0\), indicating the stability of the reconstruction with
respect to noise.

\begin{figure}[H]
  \centering  \includegraphics[width=0.94\linewidth]{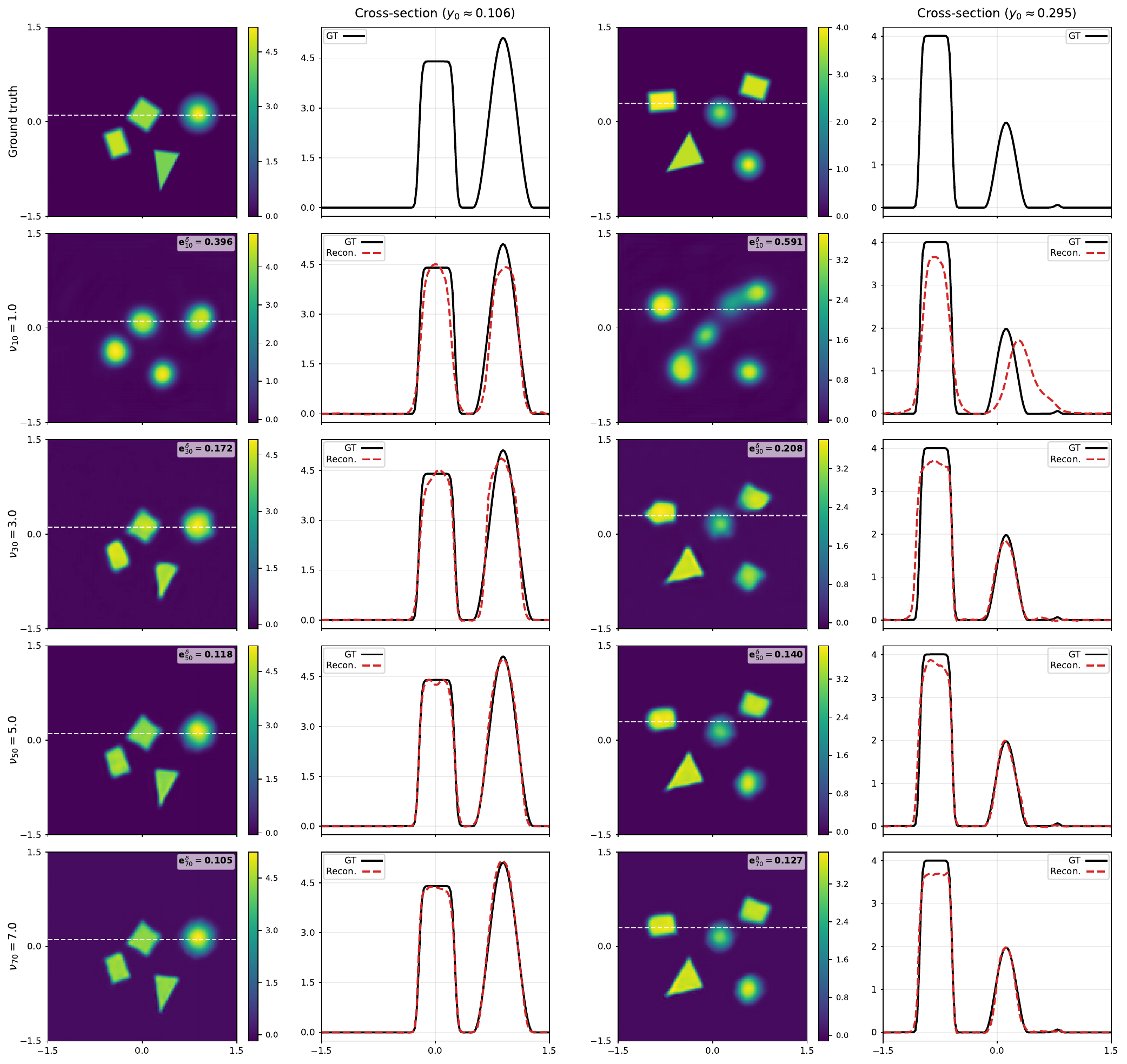}
\caption{Evolution of reconstructions and cross-sectional plots along \(y=y_0\)
for two test samples with \(\delta=0.05\). From top to bottom:
\(\nu_n=1,3,5,7\), corresponding to \(n=10,30,50,70\), respectively.}
  \label{fig2:cs-y}
\end{figure}
Next, we illustrate the evolution of the reconstruction with respect to the frequency by plotting the cross-section of the source function along a fixed horizontal line \(y=y_0\). As shown in Figure~\ref{fig2:cs-y}, the low-frequency reconstructions capture the main support of the source, while the higher-frequency reconstructions provide more accurate shape and amplitude information. It is seen that the relative error decreases as the frequency increases. This improvement is clearly reflected in the cross-section plots: the reconstructed traces along \(y=y_0\) gradually approach the ground truth, especially near the peaks and sharp transitions. This indicates the convergence of the proposed multi-level framework.

\smallskip
\noindent\textbf{(ii) Comparison with the Landweber method.}
We compare the reconstruction results of the proposed multi-level learning
framework with those of the Landweber iteration
(see Figures~\ref{fig:mean_err} and \ref{fig:compare}).  Using the noisy data with \(\delta=0.05\), the Landweber iteration is applied at
level \(n\) to solve the least-squares problem
\[
    \min_{v} \frac{1}{2}\sum_{m=1}^{n}
    \|A_m v-b_m^\delta\|_2^2 .
\]
Here
\(
    b_m^\delta
    =
    \bigl(
    \widehat{f}_{m}^{\delta}(\theta_1),
    \ldots,
    \widehat{f}_{m}^{\delta}(\theta_{N_\theta})
    \bigr)^\top
\)
is the vector of Fourier-type coefficients computed from the noisy boundary
data through the boundary integral relation \eqref{bound_int}, \(v\) is the
coefficient vector of the finite-dimensional approximation of the source
function, and \(A_m\) is the corresponding finite element discretization matrix at \(\kappa_m\).  Figure~\ref{fig:mean_err} plots the mean relative error
\(\bar e^\delta(\nu_n)\) as a function of \(\nu_n\), \(n=1,\ldots,N\). The mean relative error decreases for both methods as the frequency increases. However,
over the tested frequency range, the proposed method consistently yields smaller
errors than the Landweber method, indicating improved reconstruction accuracy.
\begin{figure}[H]
  \centering
\includegraphics[width=0.65\linewidth]{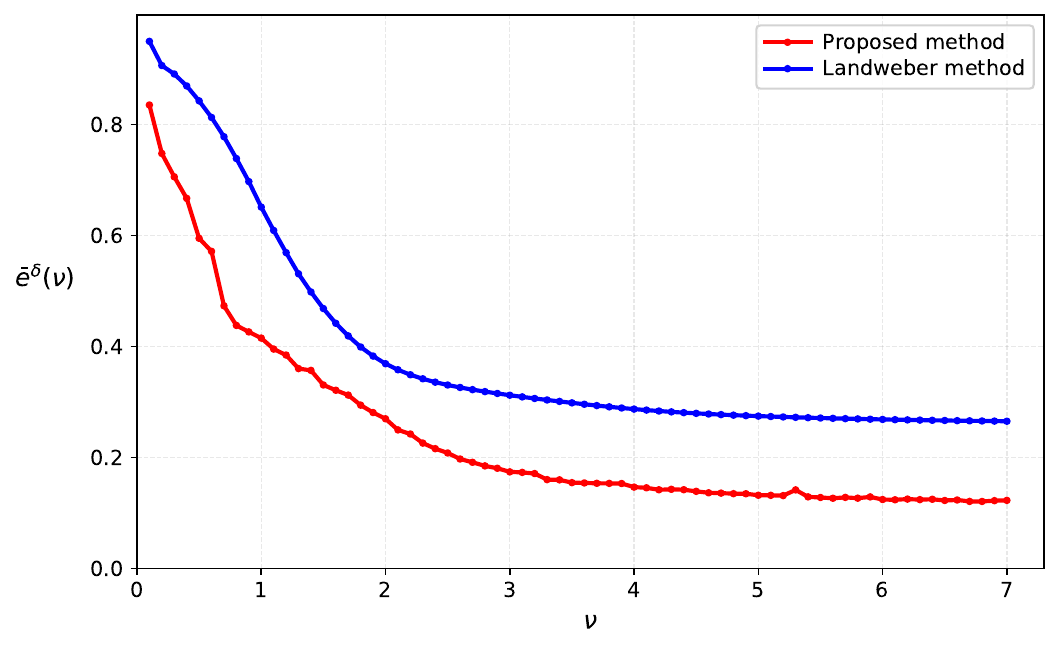}
\caption{Mean relative error \(\bar e^\delta(\nu)\) of the proposed method
and the Landweber method over the test samples as the frequency increases,
where \(\nu\in\{\nu_n\}_{n=1}^N\).}
  \label{fig:mean_err}
\end{figure}

Figure~\ref{fig:compare} shows the reconstructions at the final frequency level \(N\). The reconstruction result $f_{\mathrm{LW},N}^{\delta}$ obtained by the Landweber method recovers the main support of the source but
contains oscillatory artifacts, which are related to the limited angular sampling used in the computation of the boundary integrals in \eqref{bound_int}. In contrast, the proposed method
produces a sharper reconstruction that is closer to the ground truth. This improvement reflects the stability of the proposed multi-level framework
with respect to noise. The results suggest
that the learned model provides an implicit regularization from the training data, thereby reducing oscillatory artifacts.

\begin{figure}[H]
  \centering
\includegraphics[width=0.85\linewidth]{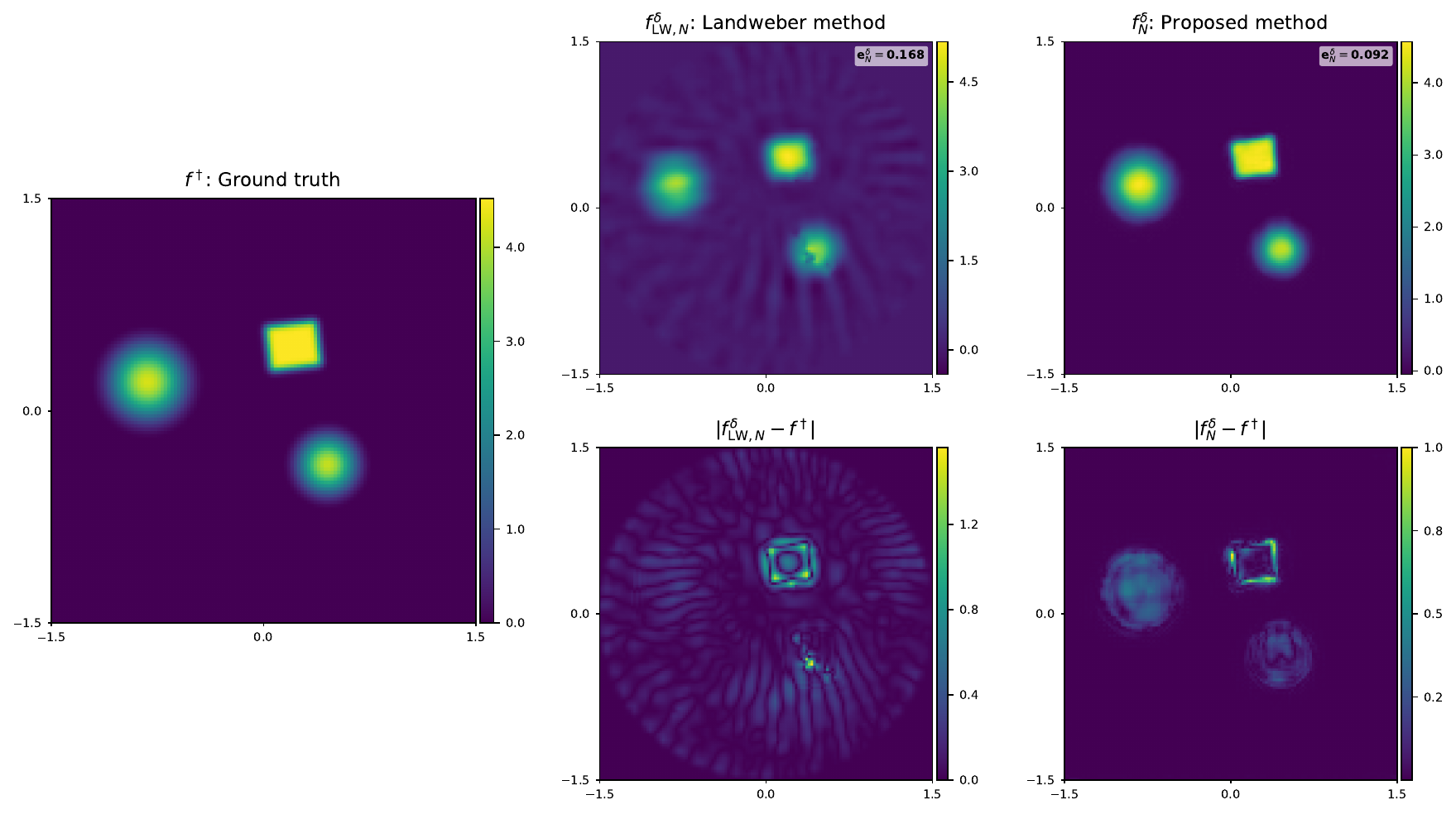}
\caption{Comparison of the Landweber method and the proposed method at the
final frequency level \(N\), with \(\nu_N=\nu_{\max}\) and \(\delta=0.05\).
First column: ground truth \(f^\dagger\). Second column: reconstruction result
\(f_{\mathrm{LW},N}^{\delta}\) obtained by the Landweber iteration and its pointwise absolute error
\(|f_{\mathrm{LW},N}^{\delta}-f^\dagger|\). Third column: reconstruction result
\(f_N^\delta\) obtained by the proposed method and its pointwise absolute error
\(|f_N^\delta-f^\dagger|\).}
  \label{fig:compare}
\end{figure}

\begin{figure}[H]
  \centering
\includegraphics[width=0.94\linewidth]{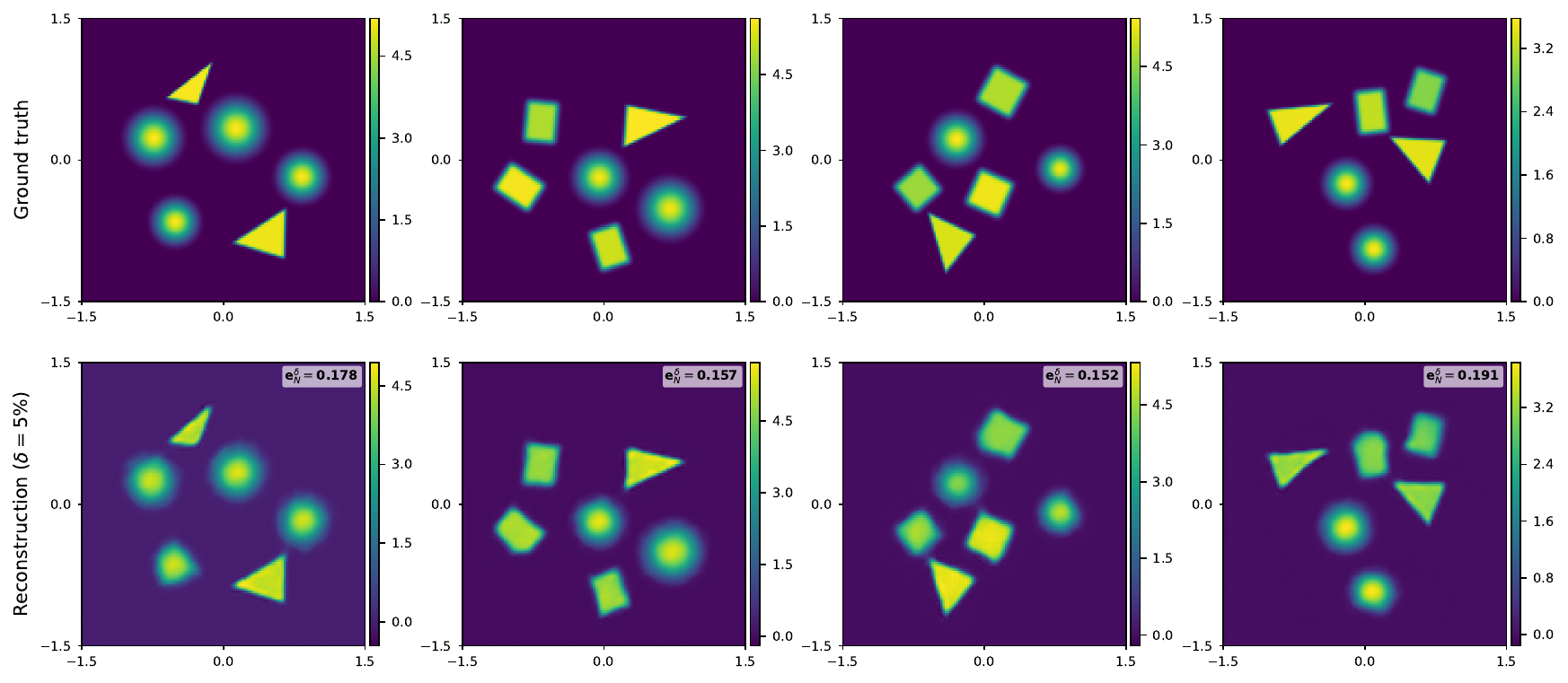}
\caption{Generalization test on samples with \(6\) inclusions. Top row:
ground-truth sources. Bottom row: reconstructions obtained at the final
frequency level \(N\), with \(\nu_N=\nu_{\max}\) and 
\(\delta=0.05\).}
  \label{fig:gen}
\end{figure}

\smallskip
\noindent\textbf{(iii) Generalization test.}
We test the trained model on samples that contain more source inclusions than
those in the training data. More precisely, the network is trained on samples with
\(1\)--\(5\) inclusions and tested on samples with \(6\) inclusions.
Figure~\ref{fig:gen} shows that the proposed method still recovers the main
support, locations, and geometric features of the sources. These results
indicate that the proposed framework remains effective for source configurations
beyond the training range.

\subsection{Inverse Medium Problem}
In this subsection, we consider the case \(\delta=0\). We set \(\nu_{\min}=0.4\), \(\nu_{\max}=6\), and the frequency step size \(\delta\nu=0.4\) so that \(N=15\).

\medskip
\noindent\textbf{Data generation and training setup.}
We generate \(5000\) synthetic medium functions for this numerical experiment. Each medium \(q^{\dagger}\) consists of 8 cosine bumps in the form of
$f(x; c, h, r) = \frac{h}{2}\left(\cos( \frac{\pi|x - c|}{r}) + 1 \right) \chi_{|x - c| < r} $. We use the radius $r\in [0.2, 0.4]$, and the intensity $h\in [0.5, 1.5]$, and the centers are uniformly distributed in $\Omega$.

For each $n=1,\ldots,N$, we use $N_\beta=32$ incident directions defined by
\[
\beta_k=\frac{2\pi(k-1)}{N_\beta},
\qquad
\mathbf p_k=(\cos\beta_k,\sin\beta_k),
\qquad k=1,\ldots,N_\beta.
\]
For the $k$-th incident direction, the incident field is given by
\[
\phi^{\mathrm{inc}}_{n,k}(\mathbf r)
=
e^{\mathrm i\kappa_n \mathbf p_k\cdot \mathbf r}.
\]
We then collect the Cauchy data of the corresponding scattered field
\(\psi_{n,k}\):
\[
D_{\kappa_n,k}
=
\left(
\psi_{n,k}|_{\partial\Omega},
\partial_{\mathbf n}\psi_{n,k}|_{\partial\Omega}
\right).
\]
Here \(\psi_{n,k}\) is obtained by numerically solving
\[
\Delta \psi_{n,k}
+
\kappa_n^2(1+q^\dagger)\psi_{n,k}
=
-\kappa_n^2 q^\dagger \phi^{\mathrm{inc}}_{n,k}
\qquad \text{in } \mathbb R^2.
\]

To construct the domain feature, we introduce \(N_\theta=8\) auxiliary
directions:
\[
\theta_j=\frac{2\pi(j-1)}{N_\theta},
\qquad
\mathbf d_j=(\cos\theta_j,\sin\theta_j),
\qquad j=1,\ldots,N_\theta ,
\]
and define
\[
\phi_{n,j}(\mathbf r)
=
e^{\mathrm i\kappa_n \mathbf d_j\cdot \mathbf r}.
\]
Using \eqref{14}, the boundary data provide the following Fourier-type
measurements of \(q^\dagger\) under the Born approximation:
\begin{equation}\label{17}
    \kappa_n^{-2}
\int_{\partial\Omega}
\left(
\psi_{n,k}\,\partial_{\mathbf n}\phi_{n,j}
-
\phi_{n,j}\,\partial_{\mathbf n}\psi_{n,k}
\right)\,\mathrm ds
\approx
\int_\Omega
q^\dagger(\mathbf r)
e^{\mathrm i\kappa_n(\mathbf d_j+\mathbf p_k)\cdot \mathbf r}
\,\mathrm d\mathbf r ,
\end{equation}
for \(k=1,\ldots,N_\beta\) and \(j=1,\ldots,N_\theta\). The domain feature \(\tilde q_n\) is then obtained by solving the resulting least-squares system.

For the training procedure in Algorithm \ref{algo}, each ground-truth medium  \(q_{[i]}^\dagger\) and its corresponding domain-defined
features \(\{\widetilde q_{n,[i]}\}_{n=1}^N\) are discretized on a \(128\times128\)
grid and stored as images. The resulting dataset consists of \(5000\) synthetic
samples, each containing one ground-truth medium and its corresponding
multi-frequency features. The dataset is partitioned into training, validation,
and test sets of sizes \(3200\), \(800\), and \(1000\), respectively. In addition, we adopt a U-Net architecture with a base number of feature channels $d=4$ and train it
using \textsc{Adam} with an initial learning rate of $10^{-4}$ and an ExponentialLR scheduler with ($\varrho = 0.99$). Training is performed with a batch size of \(50\) for \(300\) epochs.

\medskip
\noindent\textbf{Numerical results.}
At the final frequency level \(N\), where \(\nu_N=\nu_{\max}\), the overall training MSE is $1.38\times 10^{-2}$ and the test MSE is $1.58\times 10^{-2}$. A few test samples are randomly selected from the test dataset, and their reconstructed images are illustrated in Figure~\ref{fig:imp-1}. We observe that the reconstructed bumps are more accurate when they are disjoint. When the bumps overlap, the reconstruction becomes less accurate.

\begin{figure}[!htbp]
    \centering    \includegraphics[width=0.23\linewidth]{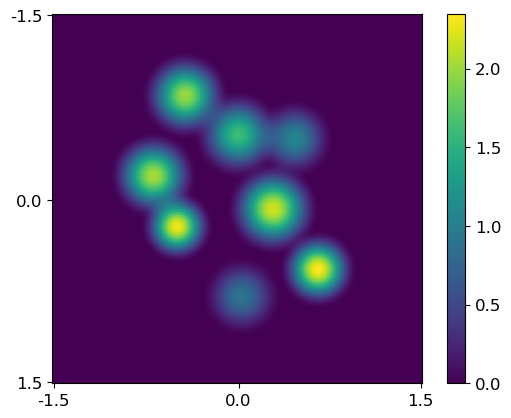}
    \includegraphics[width=0.23\linewidth]{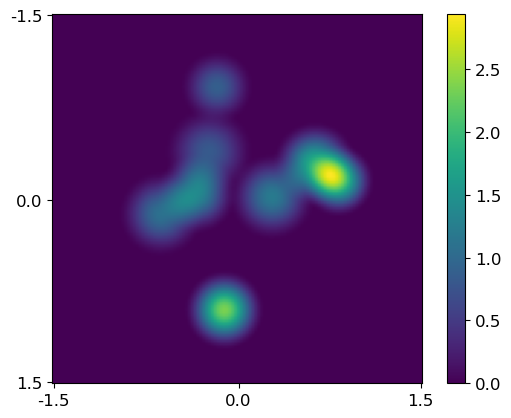}
    \includegraphics[width=0.23\linewidth]{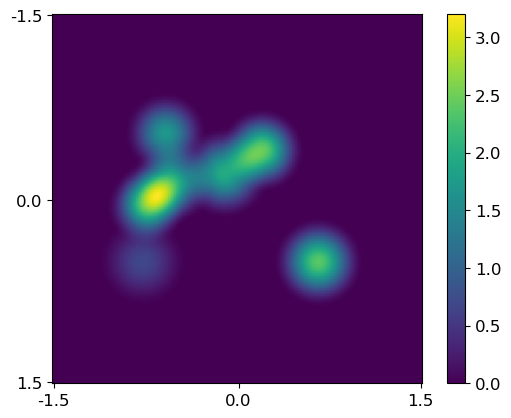}
    \includegraphics[width=0.23\linewidth]{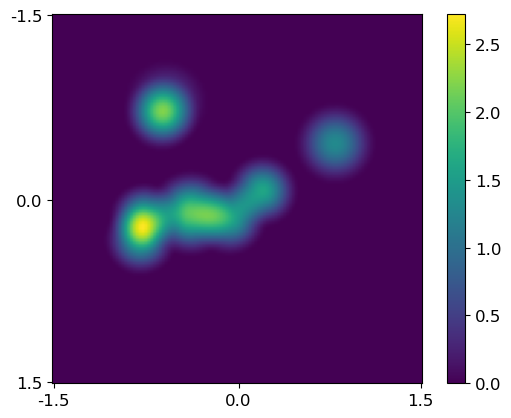}
    \\
    \includegraphics[width=0.23\linewidth]{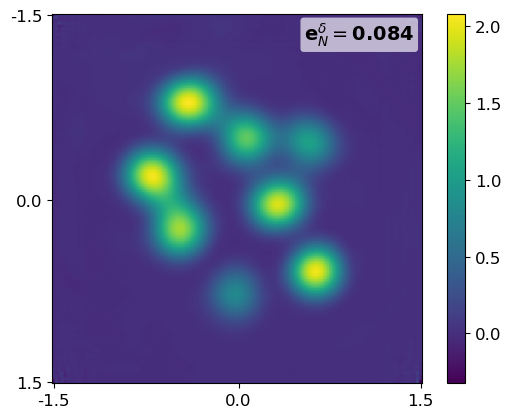}
    \includegraphics[width=0.23\linewidth]{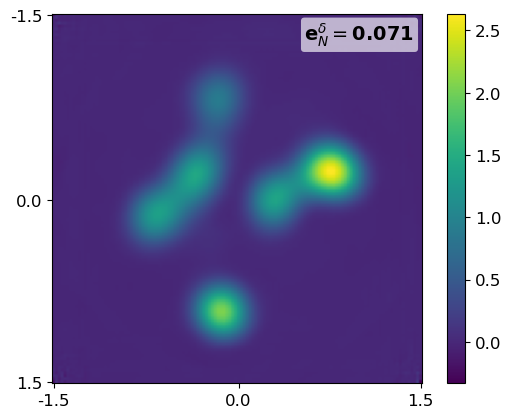}
    \includegraphics[width=0.23\linewidth]{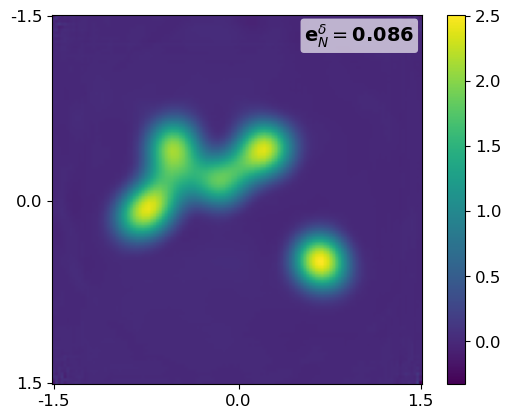}
    \includegraphics[width=0.23\linewidth]{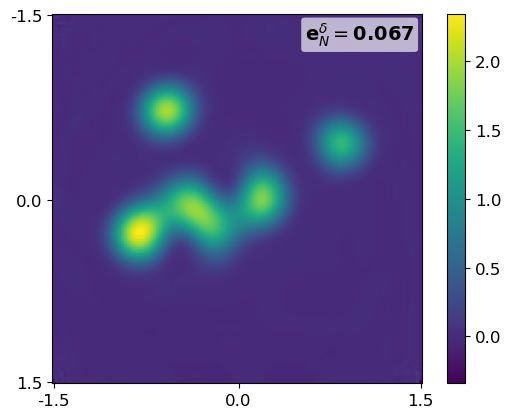}
    \caption{Final reconstructions for the IMP with \(\nu_N=\nu_{\max}\). Top row: ground truth $q^{\dagger}$. Bottom row: reconstructed medium $q_N$ at the final level \(N\).}
    \label{fig:imp-1}
\end{figure}
Figure~\ref{fig: imp-evo-low-to-high}  shows that the reconstruction for the
IMP evolves in a manner similar to that observed for the ISP. The improvements of the reconstruction are less noticeable for higher frequencies.  First, the NN uses the Born-approximation reconstruction 
\(\tilde{q}_n\), computed based on \eqref{17}, as an input at each frequency level, 
instead of following the classical recursive method~\cite{BaoLiLinTriki2015}, 
where the linearization computation has to be performed at each level. For relatively large \(\kappa_n\), the Born-approximation reconstruction \(\tilde q_n\)
already carries spectral information up to a bandwidth of \(2\kappa_n\). 
Moreover, the coupling between different frequency levels becomes stronger, 
making it more difficult for the network to extract additional information during training. Second, it is known~\cite{nagayasu2013increasing, isakov2010increasing} that the linearized reconstruction of the medium $q^{\dagger}$ can be
less stable at lower frequencies, namely, when $\kappa < \mathcal{O}(|\log \epsilon|)$, where $\epsilon$ is the error in the boundary measurements. Thus, the computation of \(\tilde{q}_n\) was performed using Tikhonov regularization, which can introduce a bias toward the smoother components.

\begin{figure}[H]
    \centering
    \includegraphics[width=0.23\linewidth]{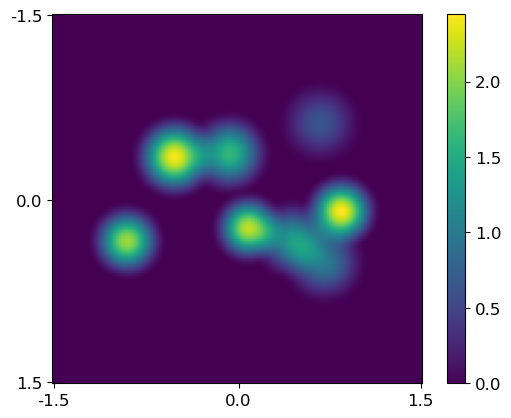}
    \includegraphics[width=0.23\linewidth]{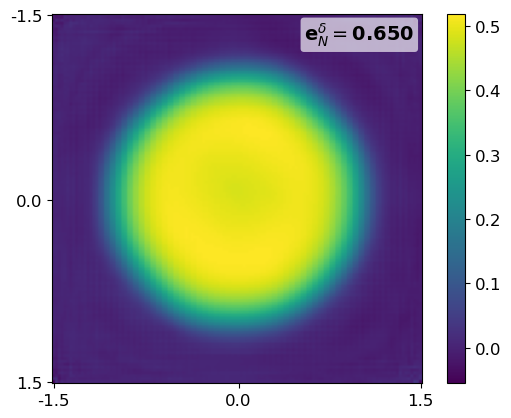}
    \includegraphics[width=0.23\linewidth]{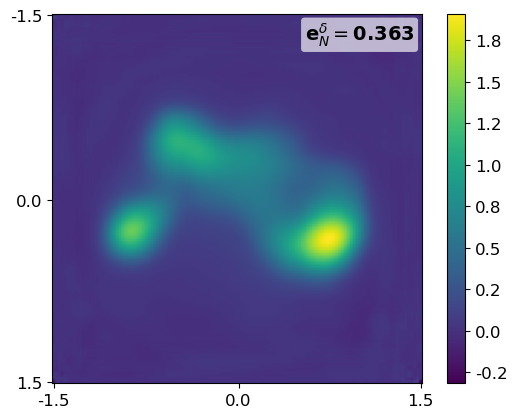}
    \includegraphics[width=0.23\linewidth]{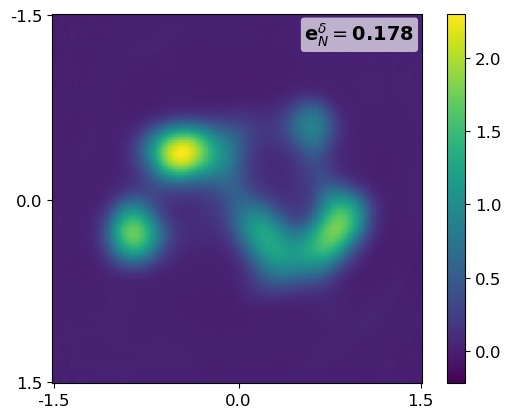}\\
    \includegraphics[width=0.23\linewidth]{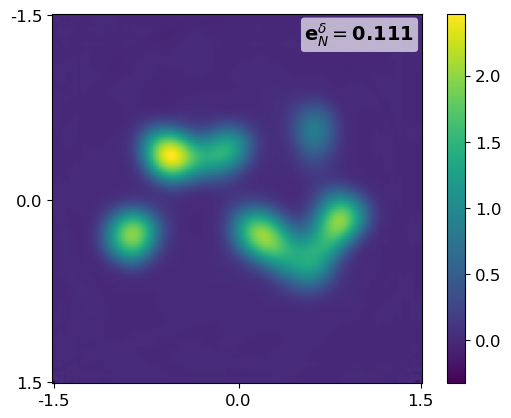}
    \includegraphics[width=0.23\linewidth]{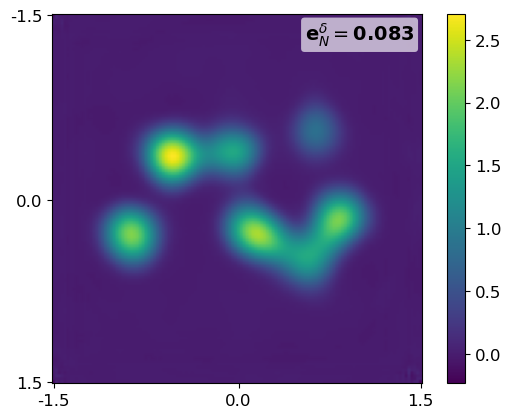}
    \includegraphics[width=0.23\linewidth]{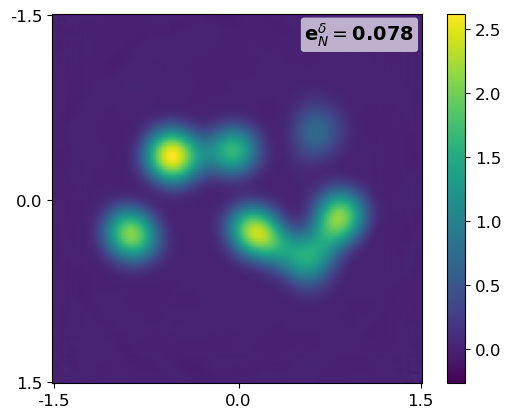}
    \includegraphics[width=0.23\linewidth]{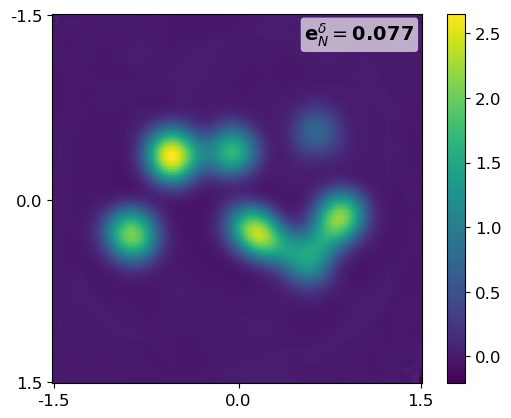}
    \caption{The first image in the top row is the ground truth \(q^{\dagger}\). The subsequent images (left to right) are reconstructions at selected levels in increasing order of \(\nu\in[\nu_{\min},\nu_{\max}]\), shown every \(0.8\) in frequency.}
    \label{fig: imp-evo-low-to-high}
\end{figure}

\section{Theoretical Analysis with Two-Layer Neural Networks}\label{sc5}
In this section, we provide a theoretical interpretation of the proposed
multi-level framework from the perspective of the neural tangent kernel (NTK), motivated by the analyses in \cite{jac2018,wang2025}. The goal is to understand how each subnetwork incorporates the
current-frequency information to update the previous reconstruction, and how these updates propagate across frequency levels. For simplicity, we perform the analysis for the ISP. The main result, stated in Theorem~\ref{th5.1}, shows that the network output at each level recovers the target Fourier components over a frequency set determined by both the current domain feature and the information inherited
from previous frequency levels. This interpretation is consistent with the progressive refinement observed in the numerical results as the frequency level increases.

We next describe the analytical setting for our main result. Assume that the Cauchy data are available at \(\{\kappa_n\}_{n=1}^{N}\), with \(\kappa_n=2\pi \nu_n\). For each \(\nu_n\), we set
\begin{equation}\label{18}
S_n
=
\Big\{
\boldsymbol{\nu}\in\mathbb R^2:\ 
\|\boldsymbol{\nu}\|=\nu_n,\,
\arg(\boldsymbol{\nu})\in
\Big\{\frac{j\pi}{N_\theta}\Big\}_{j=0}^{2N_\theta-1}
\Big\}.
\end{equation}
Here \(\arg(\cdot)\in[0,2\pi)\) denotes the polar angle. 
\(S_n\) consists of the \(2N_\theta\) uniformly sampled frequency vectors
on the circle of radius \(\nu_n\) that are used in the construction of the domain
feature. Throughout this section, for \(h\in L^2(\mathbb R^2;\mathbb R)\), we use the
Fourier transform convention
\[
\hat h(\boldsymbol{\nu})
=
\int_{\mathbb R^2}
h(\mathbf r)e^{\mathrm i2\pi\boldsymbol{\nu}\cdot\mathbf r}
\,\mathrm d\mathbf r.
\]
We use \(|z|\) for the complex modulus of a scalar \(z\), \(u^H\) for the Hermitian transpose of a complex vector \(u\), and \(\|\cdot\|=\|\cdot\|_2\).

\bigskip
\noindent\textbf{Two-layer convolutional NN.} We consider a simplified two-layer convolutional NN, motivated by the convolutional U-Net used in the numerical experiments. 
The simplified model preserves the essential input structure of the proposed framework, namely the dependence on the previous-level reconstruction and the current-frequency feature, while omitting the multi-scale encoder--decoder components. 
For \(n=1,2,\ldots,N\), the network is defined by
\begin{equation}\label{19}
\mathcal{N}_{\Theta}(\mathbf{x}_n)(\mathbf{r})
= \frac{1}{\sqrt{L}}\sum_{l=1}^L w_l \,
\sigma\left(
  \frac{1}{\sqrt{L}}\sum_{j=1}^L
  \big( (W^{(1)}_{lj} \ast \mathcal{N}_{n-1})(\mathbf{r})
  + (W^{(2)}_{lj} \ast \tilde{f}_{\,n})(\mathbf{r}) \big)
\right).
\end{equation}
The parameter set is
\[
\Theta
=
\Big\{
\{w_l\}_{l=1}^L,\,
\{W^{(1)}_{lj}\}_{1\le \ell,j\le L},\,
\{W^{(2)}_{lj}\}_{1\le \ell,j\le L}
\Big\}.
\]
The convolutional kernels
\(\{W^{(1)}_{lj}(\mathbf r,\xi)\}_{1\le \ell,j\le L}\) and
\(\{W^{(2)}_{lj}(\mathbf r,\xi)\}_{1\le \ell,j\le L}\)
are initialized as independent realizations of the random fields
\(\mathcal W^{(1)}(\mathbf r,\xi)\) and
\(\mathcal W^{(2)}(\mathbf r,\xi)\), respectively. The output weights
\(\{w_l\}_{l=1}^L\) are initialized as i.i.d. real standard normal random
variables, and \(\sigma(u)=\max\{0,u\}\) is the ReLU activation function.
Here, \(
\mathbf x_n=(\mathcal N_{n-1},\,\tilde f_{\,n})
\)
denotes the input at the \(n\)-th frequency level. 
For notational clarity in the analysis, we denote the output from the previous frequency level by $\mathcal N_{n-1}$ with \(\mathcal N_0=0\). For \(n\ge 2\),
\[
\mathcal N_{n-1}
=
\mathcal N_{\Theta_{n-1}}(\mathbf x_{n-1}),
\]
where \(\Theta_{n-1}\) denotes the trained parameters from level \(n-1\). Moreover, \(\tilde f_{\,n}\) is
the domain feature associated with the \(n\)-th frequency level, defined by
\begin{equation}\label{20}
\tilde{f}_{\,n}(\mathbf{r})
=
\sum_{\boldsymbol{\nu}\in S_n}
\hat{f}(\boldsymbol{\nu})
e^{-\mathrm{i}2\pi \boldsymbol{\nu}\cdot \mathbf{r}},
\end{equation}
where
\begin{equation}\label{21}
\phi_{\boldsymbol{\nu}}(\mathbf r)
=
e^{\mathrm{i}2\pi\boldsymbol{\nu}\cdot\mathbf r},  \quad \hat{f}(\boldsymbol{\nu})
=
\int_{\Omega}
f^\dagger(\mathbf{r})\,
\phi_{\boldsymbol{\nu}}(\mathbf{r})
\,\mathrm{d}\mathbf{r},
\quad f^\dagger\in L^2(\mathbb R^2;\mathbb R).
\end{equation}
From \eqref{11} and \eqref{bound_int},  \(\hat f(\boldsymbol{\nu})\)
can be computed from the boundary integral:
\begin{equation}
  \hat{f}(\boldsymbol{\nu})=
-\int_{\partial\Omega}
\left(
\phi_{\boldsymbol{\nu}}\,\partial_{\mathbf n}\psi
-
\psi\,\partial_{\mathbf n}\phi_{\boldsymbol{\nu}}
\right)
\,\mathrm{d}s.
\end{equation}

\bigskip
\noindent\textbf{Initialization random fields.} To track inherited and newly introduced Fourier information, we specify the random fields 
$\mathcal W^{(1)}(\mathbf r,\xi)$ and $\mathcal W^{(2)}(\mathbf r,\xi)$ through band-limited frequency sets. 
Let 
\begin{equation}\label{23}
 \mathcal{S}=\bigcup_{n=1}^N S_n,
\end{equation}
and
\[
\mathcal{V}
=
\Big\{
\boldsymbol{\nu}\in\mathbb R^2:\ 
\eta_{\min}\le\|\boldsymbol{\nu}\|<\eta_{\max}
\Big\},
\]
where \(0<\eta_{\min}<\eta_{\max}\) are chosen such that
\(
\{\nu_n\}_{n=1}^{N}\subset[\eta_{\min},\eta_{\max}].
\)
Then \(\mathcal S\subset\mathcal V\). We further partition \(\mathcal V\) into
a centrally symmetric collection of pairwise disjoint sectors. The radial grid \(\{\eta_n\}_{n=1}^{N+1}\) satisfies
\[
\eta_1=\eta_{\min},
\quad
\eta_{N+1}=\eta_{\max},
\quad
\eta_n\le\nu_n<\eta_{n+1},
\qquad n=1,\ldots,N.
\]
Define
\[
\mathcal{I}
=
\{(k,m)\in\mathbb Z^2:\ 1\le |k|\le N,\ 1\le m\le N_\theta\},
\]
together with
\[
\mathcal{I}^+
=
\{(k,m)\in\mathcal{I}:\ k\ge 1\},
\qquad
\mathcal{I}^-
=
\{(k,m)\in\mathcal{I}:\ k\le -1\}.
\]
For \((k,m)\in\mathcal I^+\), we set
\[
V_{k,m}
=
\Big\{
\boldsymbol{\nu}\in \mathcal{V}:\ 
\eta_k\le\|\boldsymbol{\nu}\|<\eta_{k+1},\
\frac{(m-1)\pi}{N_\theta}
\le
\arg(\boldsymbol{\nu})
<
\frac{m\pi}{N_\theta}
\Big\}.
\]
For the corresponding \((-k,m)\in\mathcal I^{-}\), define the centrally
symmetric sector by
\begin{equation}\label{24}
V_{-k,m}
=
-V_{k,m}
=
\{-\boldsymbol{\nu}:\boldsymbol{\nu}\in V_{k,m}\}.
\end{equation}
Thus,
\[
\mathcal{V}
=
\bigcup_{(k,m)\in\mathcal I} V_{k,m}.
\]
The band-limited random fields are then defined by
\begin{equation}\label{25}
\begin{aligned}
\mathcal W^{(1)}(\mathbf{r},\xi)
&=
C
\sum_{\boldsymbol{\nu}\in\mathcal{S}}
Z^{(1)}_{\boldsymbol{\nu}}(\xi)
e^{-\mathrm{i}2\pi\boldsymbol{\nu}\cdot\mathbf{r}},
\\
\mathcal W^{(2)}(\mathbf{r},\xi)
&=
C
\sum_{(k,m)\in\mathcal{I}}
Z^{(2)}_{k,m}(\xi)
\int_{V_{k,m}}
e^{-\mathrm{i}2\pi\boldsymbol{\nu}\cdot\mathbf{r}}
\,\mathrm{d}\boldsymbol{\nu},
\end{aligned}
\end{equation}
where $C>0$ is a fixed scaling constant used to set the variance scale of the random fields. 
The coefficients
\(\{Z^{(1)}_{\boldsymbol{\nu}}\}\) and \(\{Z^{(2)}_{k,m}\}\) are i.i.d.\
standard complex normal random variables. We impose the conjugate-symmetry
constraints
\[
Z^{(1)}_{-\boldsymbol{\nu}}
=
\overline{Z^{(1)}_{\boldsymbol{\nu}}},
\quad
Z^{(2)}_{-k,m}
=
\overline{Z^{(2)}_{k,m}},
\qquad
(k,m)\in\mathcal I^+,
\]
so that \(\mathcal W^{(1)}(\mathbf{r},\xi)\) and \(\mathcal W^{(2)}(\mathbf{r},\xi)\) are
real-valued. In the analysis, $\mathcal W^{(1)}$ tracks the Fourier components inherited from previous frequency levels through $\mathcal N_{n-1}$, while $\mathcal W^{(2)}$ organizes the current-frequency information encoded in $\tilde f_n$ into disjoint Fourier sectors.

\bigskip
\noindent\textbf{NTK analysis of the multi-level learning framework.}  At the \(n\)-th level, the loss function over \(M\) training samples is given by
\begin{equation}\label{26}
   \mathcal{T}^{(n)}(\Theta)=\frac{1}{M}\sum_{i=1}^M
     \int_{\Omega}
     \bigl| E^{(n)}_{\Theta,[i]}(\mathbf{r})|^2 \,\mathrm{d}\mathbf{r},
\end{equation}
where $E^{(n)}_{\Theta,[i]}(\mathbf{r})=\mathcal{N}_{\Theta}(\mathbf{x}_{n,[i]})(\mathbf{r})-f^\dagger_{[i]}(\mathbf{r})$ and 
\(
\mathbf x_{n,[i]}
=
\big(\mathcal N_{n-1,[i]},\tilde f_{n,[i]}\big) 
\)
is the input at level \(n\), with corresponding ground-truth source
\(f^\dagger_{[i]}\).  Here, \(\tilde f_{n,[i]}\) is the domain feature
defined in \eqref{20}.

To analyze the training dynamics, we use the NTK framework. 
In the large-width limit \(L\to\infty\) of the network \eqref{19}, the nonlinear
evolution of the network output under gradient flow can be described by a
linear kernel dynamics around initialization.
In this regime, the NTK changes only negligibly during training and, as
\(L\to\infty\), converges to a deterministic limiting kernel \cite{jac2018}.
Thus, the training process can be analyzed through the associated kernel operator, which provides a way to study how the reconstruction is updated across frequency levels.
Specifically, let \(\Theta(t)\) denote the network parameters evolving under the gradient flow
\[
\frac{d\Theta(t)}{dt}=-\nabla_{\Theta}\mathcal T^{(n)}(\Theta(t)).
\] 
In the infinite-width limit, let \(\mathcal K_n\) denote the
corresponding limiting NTK at level \(n\). Then the loss along the gradient flow satisfies
\begin{equation}\label{27}
     \frac{\mathrm{d}}{\mathrm{d}t}\mathcal{T}^{(n)}(\Theta(t))=-\frac{4}{M^2}\sum_{i=1}^M\sum_{i'=1}^M\int_{\Omega}\int_{\Omega}E^{(n)}_{\Theta(t),[i]}(\mathbf{r})  \mathcal{K}_{n}\left(\mathbf x_{n,[i]},\mathbf x_{n,[i']};\mathbf{r}, \mathbf{r}'\right))E^{(n)}_{\Theta(t),[i']}(\mathbf{r}') \, \mathrm{d}\mathbf{r}^\prime\,\mathrm{d}\mathbf{r}.
\end{equation}

Note that \(\widehat{f_{[i]}}(\boldsymbol{\nu})
=
\int_{\mathbb R^2}
f_{[i]}^\dagger(\mathbf r)e^{\mathrm i2\pi\boldsymbol{\nu}\cdot\mathbf r}
\,\mathrm d\mathbf r\). Next, we present assumptions and the main result. 
\begin{assumption}\label{a1}
 For any parameter $\Theta$ and any admissible input $\mathbf{g}$,
$\mathcal{N}_{\Theta}(\mathbf{g}) \in L^2(\mathbb{R}^2; \mathbb R)$
and has compact support contained in $\Omega$.
\end{assumption}
\begin{assumption}\label{a2}
    For each  $i \in \{1,2, \ldots, M \}$, 
  \(|\widehat{f_{[i]}}|^2\not\equiv 0\) on \(S_n\), where \(S_n\) is defined
in \eqref{18}.
\end{assumption}
\begin{assumption}\label{a3}
    For each $n \in \{1,2, \ldots,N\}$, there exists an integer $\tau_n \ge 1$ such that
    the family $\{|\widehat{f_{[i]}}(\boldsymbol{\nu})|^{2\tau_n}\}_{i=1}^M$ is linearly
    independent on $S_n$. Namely, for any constants \(c_i \in \mathbb C\), if
    \[
        \sum_{i=1}^M c_i\,|\widehat{f_{[i]}}(\boldsymbol{\nu})|^{2\tau_n} = 0
        \quad \text{for all } \boldsymbol{\nu} \in S_n,
    \]
    then $c_i = 0$, $i =1,\ldots,M$.
\end{assumption}
Without loss of generality, we assume that the frequency set $S_n$ contains at least as many elements as the training set, i.e., $|S_n|\ge M$.
We set
\[
S_{n,[i]}=\{\boldsymbol{\nu}: \boldsymbol{\nu}\in S_n \text{ such~that } \widehat{f_{[i]}}(\boldsymbol{\nu})\neq 0\}, \quad \tilde{S}_{n-1,[i]}=\{\boldsymbol{\nu}: \boldsymbol{\nu}\in \mathcal{S} \text{ such~that }  \widehat{\mathcal{N}_{n-1,[i]}}(\boldsymbol{\nu})\neq 0\},
\]
where \(\mathcal S\) is defined in \eqref{23}. 
The following theorem shows that, at each level, the Fourier transform
of the NN output agrees with that of the target function
on the set of frequency vectors generated by finite additive combinations
of the input frequency points.
\begin{theorem}\label{th5.1}
Under Assumptions~\ref{a1}--\ref{a3}, if $\Theta(t)$ converges to $\Theta_n$ as $t\to\infty$ and $\nabla_\Theta \mathcal T^{(n)}(\Theta_n)=\mathbf{0}$,
then, for $i=1,\ldots,M$,
\[
\widehat{\mathcal{N}_{n,[i]}}(\boldsymbol{\xi})=\widehat{f_{[i]}}(\boldsymbol{\xi}),
\qquad \forall\,\boldsymbol{\xi}\in B_{n,[i]},
\]
where
\[
B_{n,[i]}
=
\bigcup_{\substack{
\boldsymbol{\alpha}\in\mathbb Z_{\ge 0}^3\\
\gamma(\boldsymbol{\alpha})\in \mathbb Z_{\geq 0}
}}
\left\{
\boldsymbol{\xi}\in\mathbb R^2:
\boldsymbol{\xi}
=
\sum_{p=1}^{\gamma(\boldsymbol{\alpha})}\boldsymbol{\nu}^{(p)}
+
\sum_{p=1}^{\widetilde \gamma(\boldsymbol{\alpha})}\boldsymbol{w}^{(p)},
\quad
\boldsymbol{\nu}^{(p)}\in S_{n,[i]},\ 
\boldsymbol{w}^{(p)}\in \widetilde S_{n-1,[i]}
\right\},
\]
with, for \(\boldsymbol{\alpha}=(\alpha^{(1)},\alpha^{(2)},\alpha^{(3)})\),
\(
\gamma(\boldsymbol{\alpha})=2\alpha^{(1)}+\alpha^{(2)}-2\tau_n
\), and \(
\widetilde \gamma (\boldsymbol{\alpha})=\alpha^{(2)}+2\alpha^{(3)}
\).  Moreover,  \(\mathbb Z_{\geq 0}:=\{0,1,2,\ldots\}\), and 
\(\mathbb Z_{\geq 0}^3=(\mathbb Z_{\geq 0})^3\) denotes the set of three-dimensional vectors with nonnegative integer entries. Here an empty sum is interpreted as zero. 
\end{theorem}

The proof of Theorem~\ref{th5.1} relies on the following two auxiliary lemmas.
The first lemma collects well-known formulas; see, e.g., \cite{cho2009} and standard formulas for the quadrant probability of a bivariate normal distribution. 

\begin{lemma}\label{le1}
Let \((Y_1,Y_2)\) be a centered Gaussian random vector with covariance matrix
\[
\Sigma=
\begin{pmatrix}
\Sigma_{11} & \Sigma_{12}\\
\Sigma_{12} & \Sigma_{22}
\end{pmatrix},
\]
and let \(\sigma(z)=\max\{0,z\}\).
Define
\[
\lambda_1=\sqrt{\Sigma_{11}},\qquad
\lambda_2=\sqrt{\Sigma_{22}},\qquad
\rho=\frac{\Sigma_{12}}{\lambda_1\lambda_2},\qquad
\arccos\rho=\theta\in[0,\pi].
\]
Then
\[
\mathbb{E}\!\big[\sigma(Y_1)\sigma(Y_2)\big]
=\frac{\lambda_1\lambda_2}{2\pi}\Big(\sin\theta+(\pi-\theta)\cos\theta\Big)
=\frac{\lambda_1\lambda_2}{2\pi}\Big(\sqrt{1-\rho^2}+(\pi-\arccos\rho)\,\rho\Big),
\]
and
\[
\mathbb{E}\!\big[\dot\sigma(Y_1)\dot\sigma(Y_2)\big]
=\mathbb{P}(Y_1>0,\,Y_2>0)
=\frac{1}{4}+\frac{1}{2\pi}\arcsin\rho
=\frac{1}{2}-\frac{1}{2\pi}\arccos\rho .
\]
\end{lemma}
Let \(\mathcal{K}_{n,[i],[i']}(\mathbf{r}, \mathbf{r}')=\mathcal{K}_{n}\left(\mathbf x_{n,[i]},\mathbf x_{n,[i']};\mathbf{r}, \mathbf{r}'\right)\). The following lemma gives the explicit form of \(\mathcal{K}_{n,[i],[i']}\) in \eqref{27}; its proof is given in Appendix~\ref{apx}.

\begin{lemma}\label{le2}
Let Assumptions~\ref{a1} and \ref{a2} hold. For  $n \in \{1,2,\ldots,N\}$, as the width $L \to \infty$ in~\eqref{19}, we have
\begin{equation}\label{28}
\begin{aligned}
\mathcal{K}_{n,[i],[i']}=\frac{C^2\|\hat{\mathbf{f}}_{n,[i]}\|_{\mathcal S}\|\hat{\mathbf{f}}_{n,[i']}\|_{\mathcal S}}{2\pi}\Big(\sqrt{1-\rho_{n,[i],[i']}^2}+2(\pi-\arccos\rho_{n,[i],[i']})\,\rho_{n,[i],[i']}\Big),
\end{aligned}
\end{equation}
where
\[
\hat{\mathbf{f}}_{n,[i]}(\boldsymbol{\nu})
=
\left(
\widehat{\mathcal{N}_{n-1,[i]}}(\boldsymbol{\nu}),
\ \widehat{f_{[i]}}(\boldsymbol{\nu})\mathds{1}_{S_n}(\boldsymbol{\nu})
\right)^{\top},\quad
\|\hat{\mathbf{f}}_{n,[i]}\|_{\mathcal S}=\sqrt{\sum_{ \boldsymbol{\nu}\in \mathcal{S}}\left(\left| \widehat{\mathcal{N}_{n-1,[i]}}(\boldsymbol{\nu})\,
 \right |^2+|\widehat{f_{[i]}}(\boldsymbol{\nu})|^2  \mathds{1}_{S_n}(\boldsymbol{\nu})\right) },
\]
$\mathds{1}_{S_n}$ is the indicator function of the set $S_n$ defined in \eqref{18}, and \(\rho_{n,[i],[i']}(\mathbf{r}, \mathbf{r}')=\rho_{n}\left(\mathbf x_{n,[i]},\mathbf x_{n,[i']};\mathbf{r}, \mathbf{r}'\right))\) given by
\begin{equation}\label{29}
    \begin{aligned}
        \rho_{n}\left(\mathbf x_{n,[i]},\mathbf x_{n,[i']};\mathbf{r}, \mathbf{r}'\right))
        =\frac{\sum\limits_{ \boldsymbol{\nu}\in \mathcal{S}} \big(\hat{\mathbf{f}}_{n,[i']}(\boldsymbol{\nu})\big)^H\hat{\mathbf{f}}_{n,[i]}(\boldsymbol{\nu}) e^{-\mathrm{i}2\pi \boldsymbol{\nu}\cdot\left(\mathbf{r}-\mathbf{r}^\prime\right)}}{\|\hat{\mathbf{f}}_{n,[i]}\|_{\mathcal S}\|\hat{\mathbf{f}}_{n,[i']}\|_{\mathcal S}}.
    \end{aligned}
\end{equation}
\end{lemma}

\bigskip
\noindent\textbf{Proof of Theorem~\ref{th5.1}.}
Using Lemma \ref{le2} and the Taylor series of $\mathcal{K}_{n,[i],[i']}$, for \(|\rho_{n,[i],[i']}(\mathbf{r}, \mathbf{r}')|\leq 1\), we have
    \begin{equation}\label{30}
         \mathcal{K}_{n,[i],[i']}=  \frac{C^2\|\hat{\mathbf{f}}_{n,[i]}\|_{\mathcal S}\|\hat{\mathbf{f}}_{n,[i']}\|_{\mathcal S}}{2\pi}\left(1+\pi\rho_{n,[i],[i']}+\sum_{\alpha=1}^{\infty}
\frac{8\binom{2\alpha-2}{\alpha-1}-\binom{2\alpha}{\alpha}}{4^{\alpha}(2\alpha-1)}\rho_{n,[i],[i']}^{2\alpha}\right).
    \end{equation}
By Assumption~\ref{a2}, \eqref{30}, and an application of Fubini's theorem, we obtain
\begin{equation}\label{31}
\frac{\mathrm{d}}{\mathrm{d}t}\mathcal{T}^{(n)}(\Theta(t))
=-\frac{4}{M^2}\left(Q_1+Q_2+Q_3\right),
\end{equation}
where
\[
Q_1=\sum_{i=1}^M\sum_{i'=1}^M\frac{C^2\|\hat{\mathbf{f}}_{n,[i]}\|_{\mathcal S}\|\hat{\mathbf{f}}_{n,[i']}\|_{\mathcal S}}{2\pi}\int_{\Omega}\int_{\Omega}E^{(n)}_{\Theta(t),[i]}(\mathbf{r}) E^{(n)}_{\Theta(t),[i']}(\mathbf{r}') \, \mathrm{d}\mathbf{r}'\,\mathrm{d}\mathbf{r},   
\]
\[
Q_2=\sum_{i=1}^M\sum_{i'=1}^M \frac{C^2\|\hat{\mathbf{f}}_{n,[i]}\|_{\mathcal S}\|\hat{\mathbf{f}}_{n,[i']}\|_{\mathcal S}}{2}\int_{\Omega}\int_{\Omega}E^{(n)}_{\Theta(t),[i]}(\mathbf{r})\rho_{n,[i],[i']}(\mathbf{r}, \mathbf{r}')E^{(n)}_{\Theta(t),[i']}(\mathbf{r}') \, \mathrm{d}\mathbf{r}'\,\mathrm{d}\mathbf{r},
\]
and
\[Q_3=\frac{C^2}{2\pi}\sum_{\alpha=1}^{\infty}
\frac{8\binom{2\alpha-2}{\alpha-1}-\binom{2\alpha}{\alpha}}{4^{\alpha}(2\alpha-1)}\sum_{i=1}^M\sum_{i'=1}^M\|\hat{\mathbf{f}}_{n,[i]}\|_{\mathcal S}\|\hat{\mathbf{f}}_{n,[i']}\|_{\mathcal S}\int_{\Omega}\int_{\Omega}E^{(n)}_{\Theta(t),[i]}(\mathbf{r})  \rho_{n,[i],[i']}^{2\alpha}(\mathbf{r}, \mathbf{r}')E^{(n)}_{\Theta(t),[i']}(\mathbf{r}') \, \mathrm{d}\mathbf{r}'\,\mathrm{d}\mathbf{r}.
\]
It is obvious that
\begin{equation}\label{32}
    Q_1=\frac{C^2}{2\pi}\left(\sum_{i=1}^M\|\hat{\mathbf{f}}_{n,[i]}\|_{\mathcal S}\int_{\Omega}E^{(n)}_{\Theta(t),[i]}(\mathbf{r})\mathrm{d}\mathbf{r}\right)^2.
\end{equation}
Using \eqref{29}, we have
\begin{equation}\label{33}
\begin{aligned}
 Q_2&=\frac{C^2}{2}\sum_{i=1}^M\sum_{i'=1}^M \sum\limits_{ \boldsymbol{\nu}\in \mathcal{S}} \Big(\widehat{\mathcal{N}_{n-1,[i]}}(\boldsymbol{\nu})\overline{\widehat{E^{(n)}_{\Theta(t),[i]}}(\boldsymbol{\nu})}\overline{\widehat{\mathcal{N}_{n-1,[i']}}(\boldsymbol{\nu})\overline{\widehat{E^{(n)}_{\Theta(t),[i']}}(\boldsymbol{\nu})}}\\
&\quad+\widehat{f_{[i]}}(\boldsymbol{\nu})\overline{\widehat{E^{(n)}_{\Theta(t),[i]}}(\boldsymbol{\nu})}\overline{\widehat {f_{[i']}}(\boldsymbol{\nu})\overline{\widehat{E^{(n)}_{\Theta(t),[i']}}(\boldsymbol{\nu})}}  \mathds{1}_{S_{n}}(\boldsymbol{\nu}) \Big)\\
&=\frac{C^2}{2} \sum\limits_{ \boldsymbol{\nu}\in \mathcal{S}}\left(\left|\sum_{i=1}^M\widehat{\mathcal{N}_{n-1,[i]}}(\boldsymbol{\nu})\overline{\widehat{E^{(n)}_{\Theta(t),[i]}}(\boldsymbol{\nu})}\right|^2+\left|\sum_{i=1}^M\widehat{f_{[i]}}(\boldsymbol{\nu})\overline{\widehat{E^{(n)}_{\Theta(t),[i]}}(\boldsymbol{\nu})}\mathds{1}_{S_{n}}(\boldsymbol{\nu})\right|^2\right)\\
&= \frac{C^2}{2} \sum\limits_{ \boldsymbol{\nu}\in \mathcal{S}}\left\| \sum_{i=1}^M \overline{\widehat{E^{(n)}_{\Theta(t),[i]}}(\boldsymbol{\nu})}\hat{\mathbf{f}}_{n,[i]}(\boldsymbol{\nu})\right\|^2,
\end{aligned}
\end{equation}
where $\widehat{E^{(n)}_{\Theta(t),[i]}}(\boldsymbol{\nu})=\int_{\Omega}E^{(n)}_{\Theta(t),[i]}(\mathbf{r}) \,e^{\mathrm{i} 2\pi \boldsymbol{\nu} \cdot \mathbf{r}} \,\mathrm{d}\mathbf{r}$. 

For $Q_3$, from \eqref{29}, we have 
\[
 \rho_{n,[i],[i']}^{2\alpha}= \frac{\sum\limits_{\Lambda}   \prod\limits_{p=1}^\alpha \left[ \big(\hat{\mathbf{f}}_{n,[i']}(\boldsymbol{\nu}^{(p)})\big)^H \hat{\mathbf{f}}_{n,[i]}(\boldsymbol{\nu}^{(p)}) \right] \left[ \big(\hat{\mathbf{f}}_{n,[i']}(\boldsymbol{\nu}'^{(p)})\big)^H \hat{\mathbf{f}}_{n,[i]}(\boldsymbol{\nu}'^{(p)}) \right]\exp\Big(-\mathrm{i}2\pi \big( \sum\limits_{p=1}^\alpha (\boldsymbol{\nu}^{(p)} + \boldsymbol{\nu}'^{(p)}) \big) \cdot (\mathbf{r}-\mathbf{r}')\Big)}{{\|\hat{\mathbf{f}}_{n,[i]}\|^{2\alpha}_{\mathcal S}\|\hat{\mathbf{f}}_{n,[i']}\|^{2\alpha}_{\mathcal S}}}
\]
where
\[
\Lambda
=\left\{(\boldsymbol{\nu}^{(1)},\ldots,\boldsymbol{\nu}^{(\alpha)},
\boldsymbol{\nu}'^{(1)},\ldots,\boldsymbol{\nu}'^{(\alpha)}): 
\boldsymbol{\nu}^{(p)},\ \boldsymbol{\nu}'^{(p)}\in \mathcal{S}, \quad p=1,\ldots,\alpha \right\}.
\]
Using the property of the Kronecker product:
$$\prod_{k=1}^n (\mathbf{a}_k^H \mathbf{b}_k) = \left( \mathbf{a}_1 \otimes \mathbf{a}_2 \otimes \cdots \otimes \mathbf{a}_n \right)^H \left( \mathbf{b}_1 \otimes \mathbf{b}_2 \otimes \cdots \otimes \mathbf{b}_n \right)=\left(\bigotimes_{k=1}^n \mathbf a_k\right)^H
\left(\bigotimes_{k=1}^n \mathbf b_k\right),$$
we have
\begin{equation}\label{34}
    \begin{aligned}
  \rho_{n,[i],[i']}^{2\alpha}=  \frac{\sum\limits_{\Lambda} \left(\Phi^\alpha_{n,[i']}(\boldsymbol{\nu},\boldsymbol{\nu}')\right)^H\Phi^\alpha_{n,[i]}(\boldsymbol{\nu},\boldsymbol{\nu}')  \exp\Big(-\mathrm{i}2\pi \big( \sum\limits_{p=1}^\alpha (\boldsymbol{\nu}^{(p)} + \boldsymbol{\nu}'^{(p)}) \big) \cdot (\mathbf{r}-\mathbf{r}')\Big)}{{\|\hat{\mathbf{f}}_{n,[i]}\|^{2\alpha}_{\mathcal S}\|\hat{\mathbf{f}}_{n,[i']}\|^{2\alpha}_{\mathcal S}}},
     \end{aligned} 
\end{equation}
where $$\Phi^\alpha_{n,[i]}(\boldsymbol{\nu},\boldsymbol{\nu}')=\bigotimes\limits_{p=1}^{\alpha}\Big(\hat{\mathbf{f}}_{n,[i]}(\boldsymbol{\nu}^{(p)})\otimes\hat{\mathbf{f}}_{n,[i]}(\boldsymbol{\nu}'^{(p)})\Big).$$
Therefore, from \eqref{34}, we can derive that
\begin{equation}\label{35}
    \begin{aligned}
Q_3&=\frac{C^2}{2\pi}\sum_{\alpha=1}^{\infty}
\frac{8\binom{2\alpha-2}{\alpha-1}-\binom{2\alpha}{\alpha}}{4^{\alpha}(2\alpha-1)}\sum\limits_{\Lambda}\left\|\sum_{i=1}^M \frac{\overline{\widehat{E^{(n)}_{\Theta(t),[i]}}\Big(\sum\limits_{p=1}^\alpha \big(\boldsymbol{\nu}^{(p)} + \boldsymbol{\nu}'^{(p)}\big) \Big)}}{\|\hat{\mathbf{f}}_{[n,[i]}\|^{2\alpha-1}_{\mathcal S}}\Phi^\alpha_{n,[i]}(\boldsymbol{\nu},\boldsymbol{\nu}')\right\|^2.
    \end{aligned}
\end{equation}
Since \(\nabla_\Theta \mathcal T^{(n)}(\Theta_n)=\mathbf{0}\), it follows from \eqref{31},
\eqref{32}, \eqref{33}, and \eqref{35} that, for any
\(\alpha^{(1)},\alpha^{(2)},\alpha^{(3)}\in\mathbb Z_{\geq 0}\) satisfying
\[
\alpha^{(1)}+\alpha^{(2)}+\alpha^{(3)}=\alpha\geq 1,
\]
we have
\begin{equation}\label{36}
    \sum_{i=1}^M \frac{\overline{\widehat{E^{(n)}_{\Theta_n,[i]}}\Big(\sum\limits_{p=1}^{\alpha^{(1)}} \big(\boldsymbol{\nu}^{(p)} + \boldsymbol{\nu}'^{(p)}\big) +\sum\limits_{p=1}^{\alpha^{(2)}} \big(\boldsymbol{u}^{(p)} + \boldsymbol{u}^{'(p)}\big)+\sum\limits_{p=1}^{\alpha^{(3)}} \big(\boldsymbol{w}^{(p)} + \boldsymbol{w}^{'(p)}\big)\Big)}}{\|\hat{\mathbf{f}}_{n,[i]}\|^{2\alpha-1}_{\mathcal S}} \boldsymbol{\Xi}_{[i]}=0,
\end{equation}
where $\boldsymbol{\nu}^{(p)}, \boldsymbol{\nu}'^{(p)}, \boldsymbol{u}^{(p)} \in S_n$ and $\boldsymbol{u}'^{(p)}, \boldsymbol{w}^{(p)}, \boldsymbol{w}'^{(p)}\in \mathcal{S}$ and
\[
\boldsymbol{\Xi}_{[i]}=\prod_{p=1}^{\alpha^{(1)}}\widehat{f_{[i]}}(\boldsymbol{\nu}^{(p)})\widehat{f_{[i]}}(\boldsymbol{\nu}'^{(p)})\prod_{p=1}^{\alpha^{(2)}}\widehat{f_{[i]}}(\boldsymbol{u}^{(p)}) \widehat{\mathcal{N}_{n-1,[i]}}(\boldsymbol{u}'^{(p)})\prod_{p=1}^{\alpha^{(3)}}\widehat{\mathcal{N}_{n-1,[i]}}(\boldsymbol{w}^{(p)})\widehat{\mathcal{N}_{n-1,[i]}}(\boldsymbol{w}'^{(p)}).
\]

By Assumption \ref{a3}, there exists a positive integer $	\tau_n$ such that $\{|\widehat{f_{[i]}}(\boldsymbol{\nu})|^{2\tau_n}\}_{i=1}^M$ is linearly independent on $S_n$.  Let \(\alpha^{(1)},\alpha^{(2)},\alpha^{(3)}\in \mathbb Z_{\geq 0}\) satisfy
\[
\alpha=\alpha^{(1)}+\alpha^{(2)}+\alpha^{(3)}
\]
and 
\[
2\alpha^{(1)}+\alpha^{(2)}=2\tau_n+\gamma
\]
for some \(\gamma\in \mathbb Z_{\geq 0}\). Then, for any fixed
\(\boldsymbol{\nu}^{(p)}\in S_n\) and
\(\boldsymbol{w}^{(p)}\in\mathcal S\), based on \eqref{36} we have
\begin{equation*}
     \sum_{i=1}^M \frac{\overline{\widehat{E^{(n)}_{\Theta_n,[i]}}\Big(\sum\limits_{p=1}^{\gamma}\boldsymbol{\nu}^{(p)} +\sum\limits_{p=1}^{\widetilde{\gamma}} \boldsymbol{w}^{(p)}\Big)}}{\|\hat{\mathbf{f}}_{n,[i]}\|^{2\alpha-1}_{\mathcal S}} \prod_{p=1}^\gamma\widehat{f_{[i]}}(\boldsymbol{\nu}^{(p)})\mathds{1}_{S_{n}}(\boldsymbol{\nu}^{(p)})\prod_{p=1}^{\widetilde{\gamma}}\widehat{\mathcal{N}_{n-1,[i]}}(\boldsymbol{w}^{(p)})|\widehat{f_{[i]}}(\boldsymbol{\nu})|^{2\tau_n}\mathds{1}_{S_{n}}(\boldsymbol{\nu})=0, 
\end{equation*}
where $\widetilde{\gamma}= \alpha^{(2)}+ 2\alpha^{(3)}$ and $\boldsymbol{\nu} \in \mathcal{S}$. Therefore, we have
\begin{equation}\label{37}
\overline{\widehat{E^{(n)}_{\Theta_n,[i]}}\Big(\sum\limits_{p=1}^{\gamma}\boldsymbol{\nu}^{(p)} +\sum\limits_{p=1}^{\widetilde{\gamma}} \boldsymbol{w}^{(p)}\Big)}\prod_{p=1}^\gamma\widehat{f_{[i]}}(\boldsymbol{\nu}^{(p)})\mathds{1}_{S_{n}}(\boldsymbol{\nu}^{(p)})\prod_{p=1}^{\widetilde{\gamma}}\widehat{\mathcal{N}_{n-1,[i]}}(\boldsymbol{w}^{(p)})=0.
\end{equation}
The conclusion follows directly from \eqref{37}.

\section{Conclusion}\label{sc6}
In this work, we proposed a multi-level learning framework for multi-frequency inverse scattering problems by integrating frequency-continuation ideas with NNs. 
The proposed NN-assisted method updates the reconstruction recursively from low to high frequencies, with a subnetwork added to the existing model at each frequency level. 
At each level, the subnetwork uses the previous-level output and a current-frequency feature constructed from the measurement data to learn a local refinement. 
In this way, the global inverse learning problem is decomposed into a sequence of recursively connected local learning tasks along the frequency axis, allowing frequency information to be incorporated progressively and promoting stability in the overall reconstruction process. 
As the frequency increases, higher Fourier components are gradually recovered, and the reconstruction is refined from coarse to fine. 
Numerical experiments demonstrate that the proposed method effectively recovers high-frequency details and improves the reconstruction accuracy, even in the presence of measurement noise. 
We also provided an NTK-based interpretation of the framework, which explains the observed progressive refinement by showing how inherited information from previous levels and new information from the current frequency level contribute to the reconstruction update.

\section*{Appendix. Proof of Lemma~\ref{le2}}\label{apx}
\begin{proof}[Proof of Lemma \ref{le2}]
 From \eqref{25} and Assumption \ref{a1}, we have
\begin{equation*}
\begin{aligned}
    (W^{(1)}_{lj}\ast  \mathcal{N}_{n-1,[i]})(\mathbf{r}, \xi)
    &=C \sum_{ \boldsymbol{\nu}\in \mathcal{S}} Z^{(1), lj}_{\boldsymbol{\nu}}(\xi)
 \int_{\mathbb R^2} e^{-\mathrm{i}2\pi \boldsymbol{\nu}\cdot (\mathbf{r}-\mathbf{r}')}\mathcal{N}_{n-1,[i]}(\mathbf{r}')\mathrm{d}\mathbf r'\\
      &= C \sum_{ \boldsymbol{\nu}\in \mathcal{S}} Z^{(1), lj}_{\boldsymbol{\nu}}(\xi)
  e^{-\mathrm{i}2\pi \boldsymbol{\nu}\cdot \mathbf{r}}\widehat{\mathcal{N}_{n-1,[i]}}(\boldsymbol{\nu}).
\end{aligned}
\end{equation*}   
Moreover, we have
\begin{equation*}
\begin{aligned}
       &(W^{(2)}_{lj}\ast  \tilde{f}_{n,[i]})(\mathbf{r}, \xi)\\
       &=C \sum_{(k,m) \in \mathcal{I}} Z^{(2), lj}_{k,m}(\xi)\,\int_{\mathbb R^2} \int_{V_{k,m}} e^{-\mathrm{i}2\pi \boldsymbol{\nu}\cdot (\mathbf{r}-\mathbf{r}')}\,
  \mathrm{d}\boldsymbol{\nu}\Big(\sum_{\boldsymbol{\nu}'\in S_n} \widehat{f_{[i]}}(\boldsymbol{\nu}') 
      e^{-\mathrm{i}2\pi \boldsymbol{\nu}' \cdot \mathbf{r}'}\Big)\,\mathrm{d}\mathbf r'\\
       &=C \sum_{(k,m) \in \mathcal{I}} Z^{(2), lj}_{k,m}(\xi)\sum_{\boldsymbol{\nu}'\in S_n} \widehat{f_{[i]}}(\boldsymbol{\nu}') \int_{\mathbb R^2} \int_{\mathbb{R}^2}   \mathds{1}_{V_{k,m}}(\boldsymbol{\nu})e^{-\mathrm{i}2\pi \boldsymbol{\nu}\cdot \mathbf{r}}\,e^{\mathrm{i}2\pi \boldsymbol{\nu}\cdot \mathbf{r}'}
  \mathrm{d}\boldsymbol{\nu}\,e^{-\mathrm{i}2\pi \boldsymbol{\nu}' \cdot \mathbf{r}'}\,\mathrm{d}\mathbf r'\\
       &=C \sum_{(k,m) \in \mathcal{I}} Z^{(2), lj}_{k,m}(\xi)\sum_{\boldsymbol{\nu}'\in S_n} \widehat{f_{[i]}}(\boldsymbol{\nu}')  \mathds{1}_{V_{k,m}}(\boldsymbol{\nu}')e^{-\mathrm{i}2\pi \boldsymbol{\nu}'\cdot \mathbf{r}}\\
       &=C \sum_{\boldsymbol{\nu}\in S_n} Z^{(2), lj}_{\boldsymbol{\nu}}(\xi) \, \widehat{f_{[i]}}(\boldsymbol{\nu}) \, e^{-\mathrm{i}2\pi \boldsymbol{\nu}\cdot \mathbf{r}},
\end{aligned}
\end{equation*}
where $\{Z^{(2), lj}_{\boldsymbol{\nu}}(\xi)\}$ are i.i.d.\ standard complex normal random variables and satisfy the conjugate-symmetry. 

Let 
\[
S^{+}_n=\Big\{\boldsymbol{\nu}\in\mathbb R^2:\ \|\boldsymbol{\nu}\|=\nu_n,\ 
\arg(\boldsymbol{\nu})\in\Big\{\frac{j\pi}{N_\theta}\Big\}_{j=0}^{N_\theta-1}
\Big\},\quad S^{(1),+}=\bigcup_{n=1}^NS^{+}_n.
\]
We rewrite that
\begin{equation*}
    \begin{aligned}
(W^{(1)}_{lj}\ast  \mathcal{N}_{n-1,[i]})(\mathbf{r}, \xi)&=C \sum_{ \boldsymbol{\nu}\in S^{(1),+}}\left( Z^{(1), lj}_{\boldsymbol{\nu}}(\xi)
  e^{-\mathrm{i}2\pi \boldsymbol{\nu}\cdot \mathbf{r}}\widehat{\mathcal{N}_{n-1,[i]}}(\boldsymbol{\nu})+\overline{ Z^{(1), lj}_{\boldsymbol{\nu}}(\xi)
  e^{-\mathrm{i}2\pi \boldsymbol{\nu}\cdot \mathbf{r}}\widehat{\mathcal{N}_{n-1,[i]}}(\boldsymbol{\nu})}\right),
      \end{aligned}
\end{equation*}
and
\begin{equation*}
    \begin{aligned}
      (W^{(2)}_{lj}\ast  \tilde{f}_{n,[i]})(\mathbf{r}, \xi)=  C \sum_{\boldsymbol{\nu}'\in S^{+}_n}\left( Z^{(2), lj}_{\boldsymbol{\nu}'}(\xi) \, \widehat{f_{[i]}}(\boldsymbol{\nu}') \, e^{-\mathrm{i}2\pi \boldsymbol{\nu}'\cdot \mathbf{r}}+\overline{Z^{(2), lj}_{\boldsymbol{\nu}'}(\xi) \, \widehat{f_{[i]}}(\boldsymbol{\nu}') \, e^{-\mathrm{i}2\pi \boldsymbol{\nu}'\cdot \mathbf{r}}}\right).
    \end{aligned}
\end{equation*}
Set
\[
X_{l,n,[i]}(\mathbf{r}, \xi)=\sum^L_{j=1}\Big((W^{(1)}_{lj}\ast \mathcal{N}_{n-1,[i]})(\mathbf{r}, \xi)+(W^{(2)}_{lj}\ast \tilde{f}_{n,[i]})(\mathbf{r}, \xi)\Big).
\]
Using the independence of $\{Z^{(2), lj}_{\boldsymbol{\nu}}\}$ and $\{Z^{(1), lj}_{\boldsymbol{\nu}}\}$ together with the identities $\mathbb{E}\left[(Z^{(q), lj}_{\boldsymbol{\nu}})^2\right]=0$ and $\mathbb{E}\left[|Z^{(q), lj}_{\boldsymbol{\nu}}|^2\right]=1$, we can derive the pre-activation covariance kernel function that
\begin{equation*}
    \begin{aligned}
     &\Sigma^{n,(1)}\left(\mathbf x_{n,[i]},\mathbf x_{n,[i']};\mathbf{r}, \mathbf{r}'\right)) \\
          &=\frac{1}{L} \mathbb{E}\left[X_{l,n,[i]}(\mathbf{r}, \xi)X_{l,n,[i']}(\mathbf{r}', \xi)\right] \\
      &=C^2 \sum_{ \boldsymbol{\nu}\in \mathcal{S}}\widehat{\mathcal{N}_{n-1,[i]}}(\boldsymbol{\nu})e^{-\mathrm{i}2\pi \boldsymbol{\nu}\cdot \mathbf{r}}\overline{\widehat{\mathcal{N}_{n-1,[i']}}(\boldsymbol{\nu})e^{-\mathrm{i}2\pi \boldsymbol{\nu}\cdot \mathbf{r}'}} +C^2 \sum_{\boldsymbol{\nu}\in S_n} \widehat{f_{[i]}}(\boldsymbol{\nu}) \, e^{-\mathrm{i}2\pi \boldsymbol{\nu}\cdot \mathbf{r}}\overline{\widehat{f_{[i']}}(\boldsymbol{\nu}) \, e^{-\mathrm{i}2\pi \boldsymbol{\nu}\cdot \mathbf{r}'}}.
    \end{aligned}
\end{equation*}
Moreover, we have the corelation function 
\begin{equation}\label{39}
    \begin{aligned}
& \rho_{n}\left(\mathbf x_{n,[i]},\mathbf x_{n,[i']};\mathbf{r}, \mathbf{r}'\right))\\
        &=\frac{\Sigma^{n,(1)}\left(\mathbf x_{n,[i]},\mathbf x_{n,[i']};\mathbf{r}, \mathbf{r}'\right))}{\sqrt{\Sigma^{n,(1)}\left(\mathbf x_{n,[i]}, \mathbf x_{n,[i]};\mathbf{r}, \mathbf{r}\right)\Sigma^{n,(1)}\left(\mathbf x_{n,[i']},\mathbf x_{n,[i']};\mathbf{r}, \mathbf{r}'\right)}}\\
        &=\frac{\sum_{ \boldsymbol{\nu}\in \mathcal{S}}\left(\widehat{\mathcal{N}_{n-1,[i]}}(\boldsymbol{\nu})e^{-\mathrm{i}2\pi \boldsymbol{\nu}\cdot \mathbf{r}}\overline{\widehat{\mathcal{N}_{n-1,[i']}}(\boldsymbol{\nu})e^{-\mathrm{i}2\pi \boldsymbol{\nu}\cdot \mathbf{r}'}} +  \widehat{f_{[i]}}(\boldsymbol{\nu}) \, e^{-\mathrm{i}2\pi \boldsymbol{\nu}\cdot \mathbf{r}}\overline{\widehat{f_{[i']}}(\boldsymbol{\nu}) \, e^{-\mathrm{i}2\pi \boldsymbol{\nu}\cdot \mathbf{r}'}} \mathds{1}_{S_n}(\boldsymbol{\nu})\right)}{\|\hat{\mathbf{f}}_{n,[i]}\|_{\mathcal S}\|\hat{\mathbf{f}}_{n,[i']}\|_{\mathcal S}}.
    \end{aligned}
\end{equation}
Combining Lemma~\ref{le1} with \eqref{39}, and following the standard NTK argument in~\cite{jac2018}, we obtain \eqref{28}. 
The details are standard and are omitted.
\end{proof}

\bibliographystyle{elsarticle-num-names}
\bibliography{reference}
\end{document}